\documentclass[reqno,a4paper,11pt]{amsart}
\usepackage[margin=1.4in]{geometry}

\usepackage{projconvex}

\def\widthhs{4}
\def\widthh{5}

\author{Quentin M\'erigot \and \'Edouard Oudet} \title{Handling
  convexity-like constraints in variational problems}

\begin{document}

\begin{abstract}
  We provide a general framework to construct finite dimen\-sio\-nal
  approximations of the space of convex functions, which also applies
  to the space of $c$-convex functions and to the space of support
  functions of convex bodies. We give precise estimates of the
  distance between the approximation space and the admissible
  set. This framework applies to the approximation of convex functions
  by piecewise linear functions on a mesh of the domain and by other
  finite-dimensional spaces such as tensor-product splines.  We show
  how these discretizations are well suited for the numerical
  solution of problems of calculus of variations under convexity
  constraints. Our implementation relies on proximal algorithms, and
  can be easily parallelized, thus making it applicable to large scale
  problems in dimension two and three.  We illustrate the versatility
  and the efficiency of our approach on the numerical solution of
  three problems in calculus of variation : 3D denoising, the
  principal agent problem, and optimization within the class of convex
  bodies.
\end{abstract}

\maketitle

\section{Introduction}

Several problems in the calculus of variations come with natural
convexity constraints. In optimal transport, Brenier theorem asserts
that every optimal transport plan can be written as the gradient of a
convex function, when the cost is the squared Euclidean distance.
Jordan, Kinderlehrer and Otto showed \cite{jordan1998variational} that
some evolutionary PDEs such as the Fokker-Planck equation can be
reformulated as a gradient flow of a functional in the space of
probability densities endowed with the natural distance constructed
from optimal transport, namely the Wasserstein space. In the
corresponding time-discretized schemes, each timestep involves the
solution of a convex optimization problem over the set of gradient
of convex functions. In a different context, the principal agent
problem proposed by Rochet and Choné \cite{rochet1998ironing} in
economy also comes with natural convexity constraints.  Despite the
possible applications, the numerical implementation of these
variational problems has been lagging behind, mainly because of a
non-density phenomenon discovered by Chon\'e and Le~Meur
\cite{chone2001non}.

\begin{figure}[t]
\centering

\begin{tabular}{c c}
\includegraphics[height=\widthhs cm]{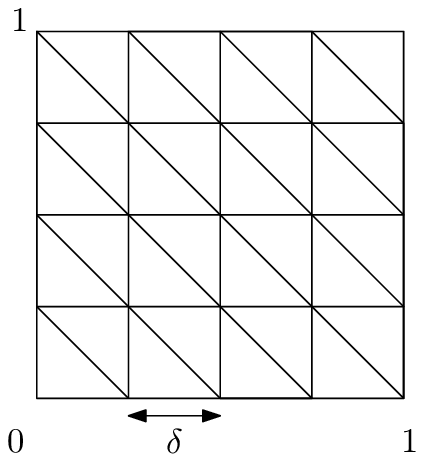}&
\includegraphics[height=\widthhs cm]{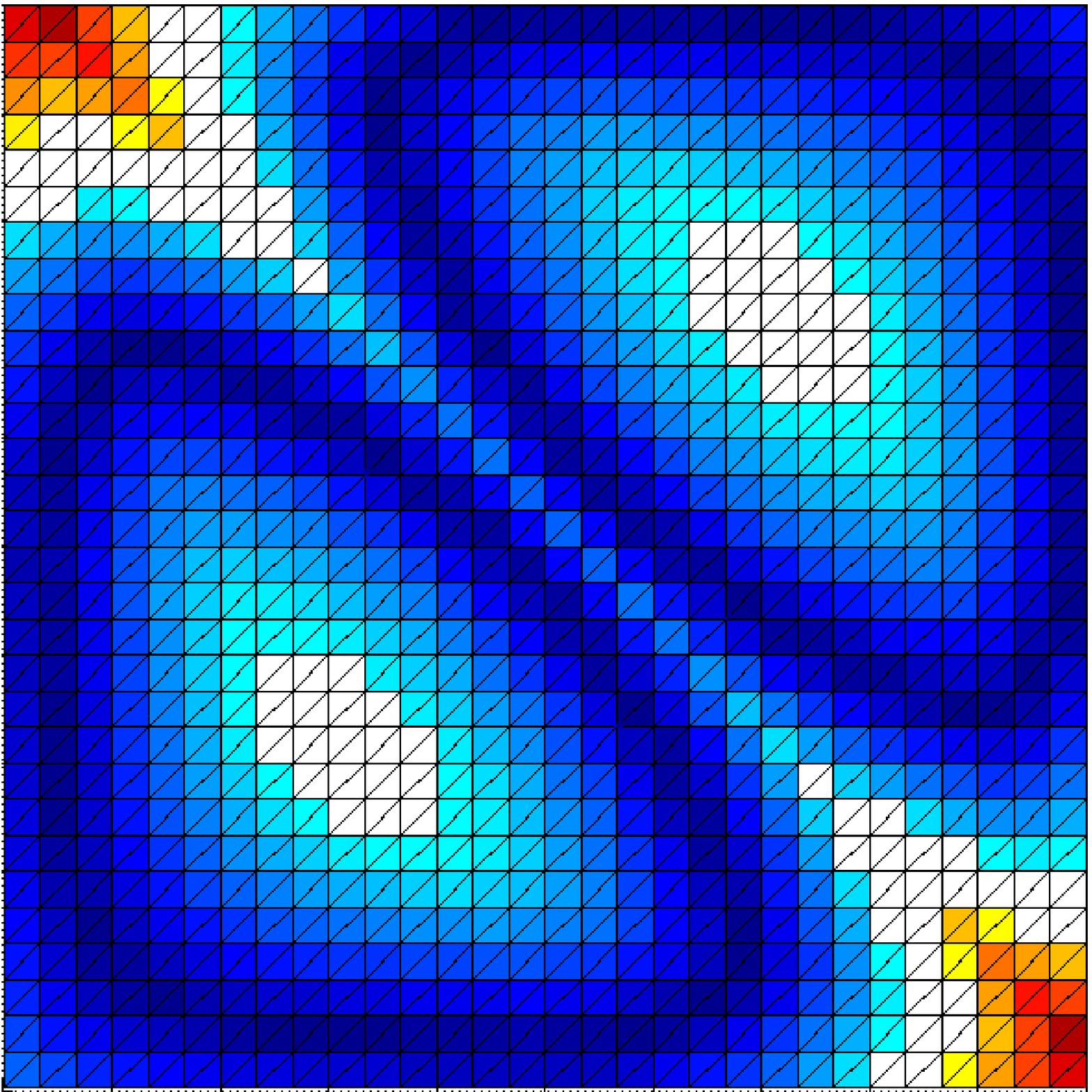} \\
\includegraphics[height=\widthhs cm]{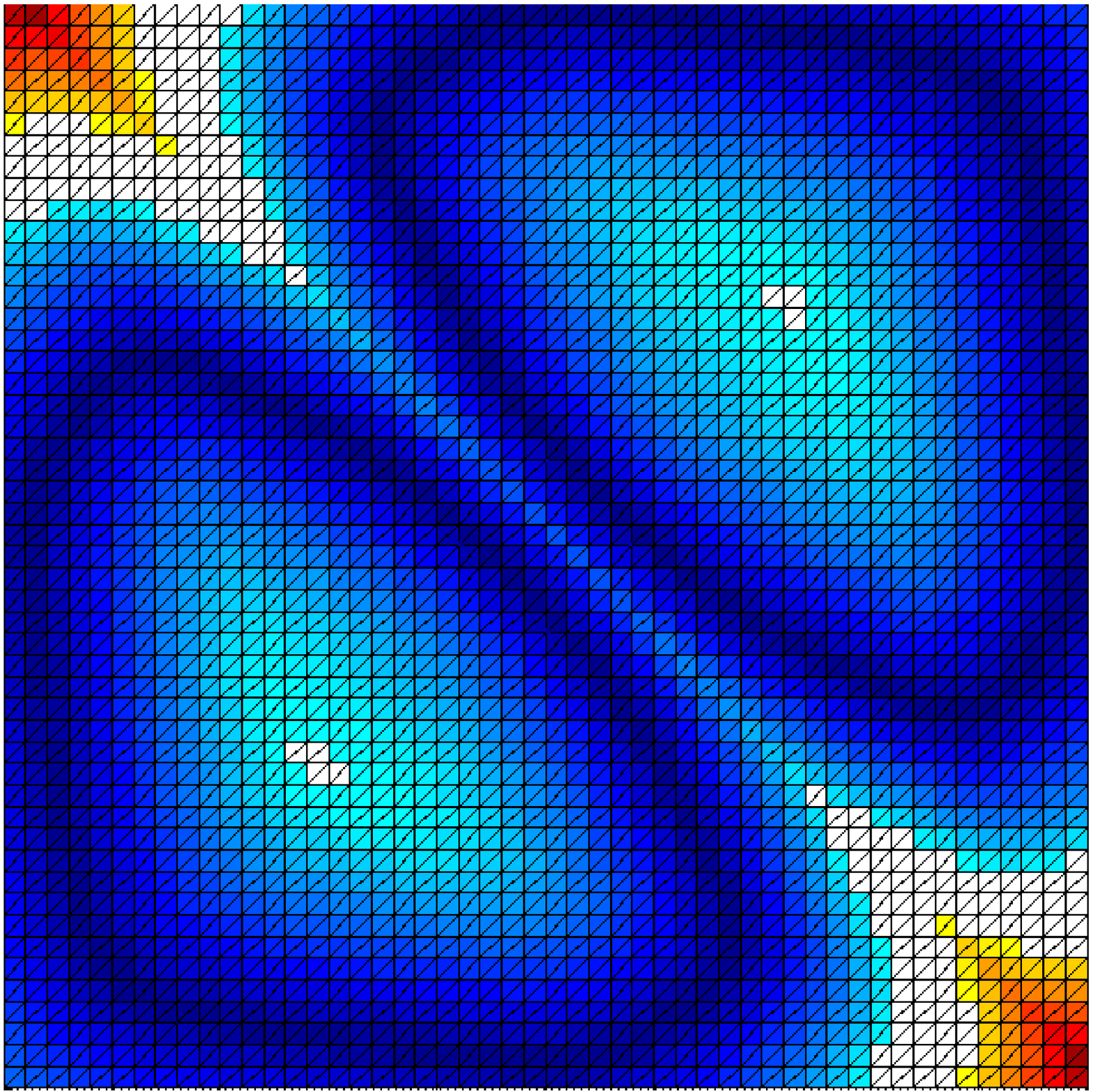}&
\includegraphics[height=\widthhs cm]{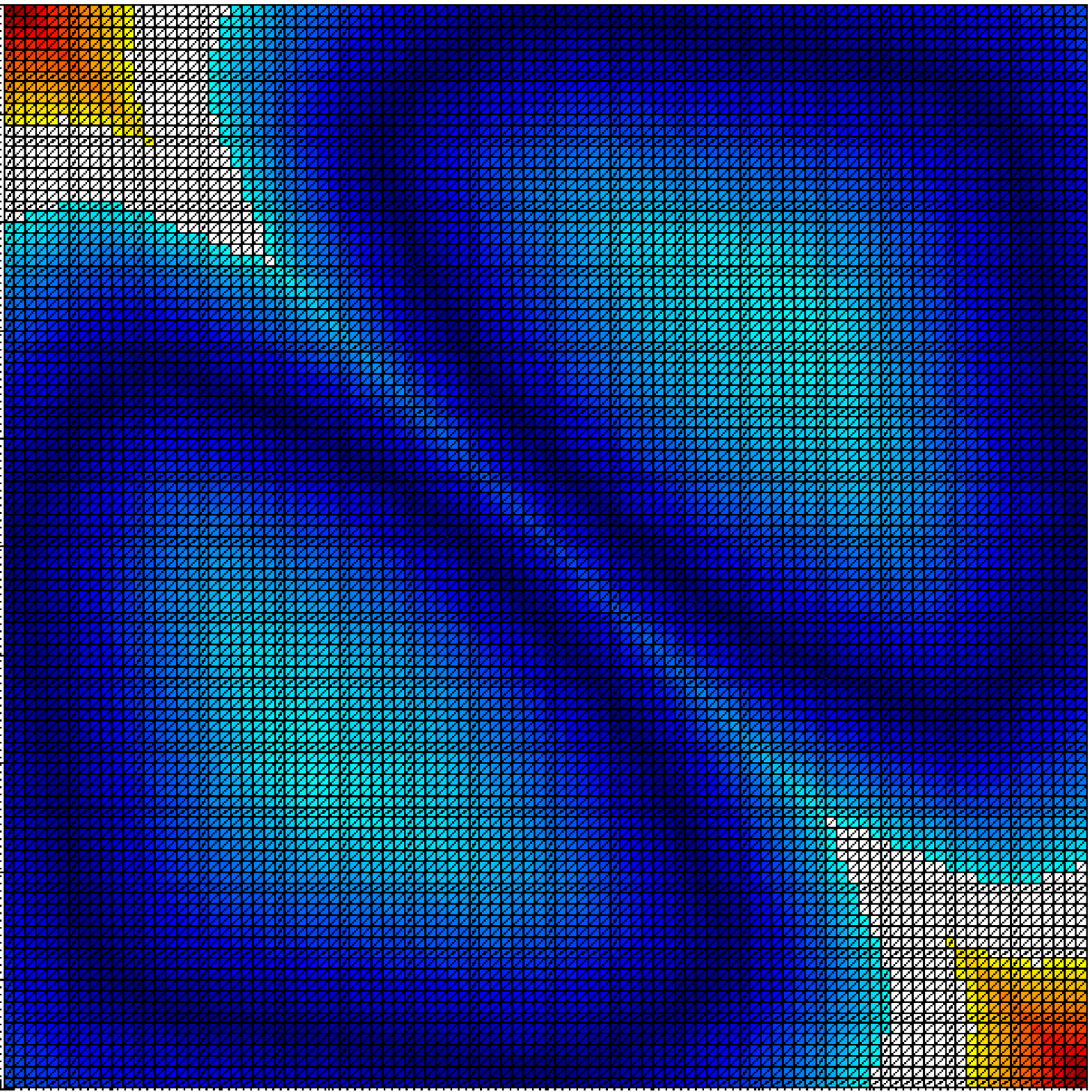}\\
\end{tabular}

\label{fig:grid}
\label{fig:chone}

\caption{Illustration of the non-density result phenomena of Chon\'e
  and Le Meur on the grid $G_\delta$ \emph{(top left)}.  We consider
  the convex function $f(x,y) = \max(0, x+y-1)$ on $[0,1]^2$ and its
  projection $g$ on the intersection $\H_{M_\eps} \cap E_\delta$,
  where $E_\delta$ is the space of piecewise linear functions on
  $G_\delta$ and $\H_{M_\eps}$ is the space of function satisfying the
  relaxed convexity constraints of
  Definition~\ref{def:discretization}. The error $\abs{f-g}$ is
  displayed for three different choices of grid size, and $\eps \ll
  \delta$. One can observe that the maximum error $\nr{f-g}_\infty$
  remains almost constant regardless of $\delta$. (In all figures, the
  upper left and lower right corners corresponds to the value $0.2$.)}
\end{figure}

Chon\'e and Le~Meur discovered that some convex functions cannot be
approximated by piecewise-linear convex functions on a regular grid
(such as the grid displayed in Figure~\ref{fig:grid}).  More
precisely, they proved that piecewise-linear convex functions on the
regular grid automatically satisfy the inequality $\frac{\partial^2
  f}{\partial x \partial y} \geq 0$ in a the sense of distributions.
Since there exists convex functions that do not satisfy this
inequality, this implies that the union of the spaces of
piecewise-linear convex functions on the regular grids
$(G_\delta)_{\delta > 0}$ is not dense in the space of convex
functions on the unit square. Moreover, this difficulty is local, and
it is likely that for any fixed sequence of meshes, one can construct
convex functions $f$ that cannot be obtained as limits of
piecewise-linear convex functions on these meshes. This phenomenon
makes it challenging to use P$^1$ finite elements to approximate the
solution of variational problems with convexity constraints.

\subsection{Related works} In this section, we briefly discuss  approaches
that have been proposed in the last decade to tackle the problem
discovered by Chon\'e and Le~Meur.

\subsubsection*{Mesh versus grid constraints}
Carlier, Lachand-Robert and Maury proposed in
\cite{carlier2001numerical} to replace the space of P$^1$ convex
functions by the space of the space of \emph{convex interpolates}. For
every fixed mesh, a piecewise linear function is a convex interpolate
if it is obtained by linearly interpolating the restriction of a
convex function to the node of the mesh. Note that these functions are
not necessarily convex, and the method is therefore not
interior. Density results are straightforward in this context but the
number of linear constraints which have to be imposed on nodes values
is rather large. The authors observe that in the case of a regular
grid, one needs $\simeq m^{1.8}$ constraints in order to describe the
space of convex interpolates, where $m$ stands for the number of nodes
of the mesh.

Aguilera and Morin \cite{aguilera2008approximating} proposed a
finite-difference approximation of the space of convex functions using
discrete convex Hessians. They prove that it is possible to impose
convexity by requiring a linear number of nonlinear constraints with
respect to the number of nodes. The leading fully nonlinear
optimization problems are solved using semidefinite programming
codes. Whereas, convergence is proved in a rather general setting, the
practical efficiency of this approach is limited by the capability of
semidefinite solvers. In a similar spirit, Oberman
\cite{oberman2013numerical} considers the space of function that
satisfy local convexity constraints on a finite set of directions. By
changing the size of the stencil, the author proposed different
discretizations which lead to exterior or interior
approximations. Estimates of the quality of the approximation are given
for smooth convex functions.

\subsubsection*{Higher order approximation by convex tensor-product splines}
An important number of publications have been dedicated in recent
years to solve shape preserving least square problems. For instance,
different sufficient conditions have been introduced to force the
convexity of the approximating functions. Whereas this problem is well
understood in dimension one, it is still an active field of research
in the context of multivariate polynomials like B\'ezier or tensor
spline functions. We refer the reader to J{\"u}ttler \cite{juttler}
and references therein for a detailed description of recent
results. In this article, J{\"u}ttler describes an interior
discretization of convex tensor-product splines. This approach is
based on the so called ``Blossoming theory'' which makes it possible
to linearize constraints on the Hessian matrix by introducing
additional variables. Based on this framework, the author illustrates
the method by computing the $L^2$ projection of some given function
into the space of convex tensor-product splines. Two major
difficulties have to be pointed out. First, the density of convex
tensor splines in the space of convex functions is absolutely non
trivial, and one may expect phenomena similar to those discovered by
Chon\'e and Le~Meur. Second, the proposed algorithm leads to a very
large number of linear constraints.

\subsubsection*{Dual approaches}
Lachand-Robert and Oudet \cite{lachand2005minimizing} developed a
strategy related to the dual representation of a convex body by its
support functions. They rely on a simple projection approach that
amounts to the computation of a convex hull, thus avoiding the need to
describe the constraints defining the set of support functions. To the
best of our knowledge, this article is the first one to attack the
question of solving problems of calculus of variations within convex
bodies. The resulting algorithm can be interpreted as a non-smooth
projected gradient descent, and gave interesting results on difficult
problems such as Newton's or Alexandrov's problems. In a similar
geometric framework, Oudet studied in \cite{oudet2013shape}
approximations of convex bodies based on Minkowski sums. It is well
known in dimension two that every convex polygon can be decomposed as
a finite sum of segments and triangles. While this result cannot be
generalized to higher dimension, this approach still allows the
generation of random convex polytopes. This process was used by the
author to study numerically two problems of calculus of variations on
the space of convex bodies with additional width constraints.

Ekeland and Moreno-Bromberg \cite{ekeland2010algorithm} proposed a
dual approach for parameterizing the space of convex functions on a
domain. Given a finite set of points $S$ in the domain, they
parameterize convex functions by their value $f_s$ \emph{and their
  gradient} $v_s$ at those points. In order to ensure that these
couples of values and gradients $(f_s, v_s)_{s\in S}$ are induced by a
convex function, they add for every pair of points in $S$ the
constraints $f_t \geq f_s + \sca{t-s}{v_s}$. This discretization is
interior, and it is easy to show that the phenomenon of Choné and Le
Meur does not occur for this type of approximation. However, the high
number of constraints makes it difficult to solve large-scale problems
using this approach. Mirebeau \cite{mirebeau2013} is currently
investigating an adaptative version of this method that would allow
its application to larger problems.

\subsection{Contributions}
We provide a general framework to construct approximations of the
space of convex functions on a bounded domain that satisfies a Lipschitz
bound. Our approximating space is a finite-dimensional polyhedron,
that is a subset of a finite-dimensional functional space that satisfies
a finite number of linear constraints. The main theoretical
contribution of this article is a bound on the (Hausdorff) distance
between the approximating polyhedron and the admissible set of convex
functions, which is summarized in Theorem~\ref{th:dh}. This bound implies in particular the density of the discretized space of functions in the space of convex functions. Our
discretization is not specific to approximation by piecewise linear
functions on a triangulation of the domain, and can easily be extended
to approximations of convex functions within other
finite-dimensional subspaces, such as the space of tensor product
splines. This is illustrated numerically in
Section~\ref{pagent_section}.

This type of discretization is well suited to the numerical solution
of problems of calculus of variations under convexity constraints. For
instance, we show how to compute the $\LL^2$ projection onto the
discretized space of convex functions in dimension $d=2,3$ by
combining a proximal algorithm~\cite{bauschke2011fixed} and an
efficient projection operator on the space of $1$D discrete convex
functions. Because of the structure of the problem, these $1$D
projection steps can be performed in parallel, thus making our
approach applicable to large scale problems in higher dimension. We
apply our non-smooth approach to a denoising problem in dimension
three in Section \ref{denoising-section}. Other problems of calculus
of variations under convexity constraints, such as the principal-agent
problem, can be solved using variants of this algorithm. This aspect
is illustrated in Section \ref{pagent_section}.

Finally, we note in Section~\ref{sec:generalization} that the
discretization of the space of convex functions we propose can be
generalized to other spaces of functions satisfying similar
constraints, such as the space of support functions of convex
bodies. The proximal algorithm can also be applied to this modified
case, thus providing the first method able to approximate the
projection of a function on the sphere onto the space of support
functions of convex bodies. Section~\ref{sec:constant} presents
numerical computations of $\LL^p$ projections (for $p=1,2,\infty$) of
the support function of a unit regular simplex onto the set of support
functions of convex bodies with constant width. We believe that these
projection operators could be useful in the numerical study of a
famous conjecture due to Bonnesen and Fenchel (1934) concerning
Meissner's convex bodies.

\subsection*{Acknowledgements} The first named author would like to
thank Robert McCann for introducing him to the principal-agent problem
and Young-Heon Kim for interesting discussions. The authors would like
to acknowledge the support of the French Agence National de la
Recherche under references ANR-11-BS01-014-01 (TOMMI) and
ANR-12-BS01-0007 (Optiform).

\subsection*{Notation} Given a metric space $X$, we denote $\C(X)$ the
space of bounded continuous functions on $X$ endowed with the norm of
uniform convergence $\nr{.}_\infty$. Every subset $L$ of the space of
affine forms on $\C(X)$ defines a convex subset $\H_L$ of the space of
continuous functions by duality, denoted by:
\begin{equation}
\label{eq:relax}
\H_L := \{ g \in \C(K);~ \forall \ell \in L,~ \ell(g) \leq 0 \}.
\end{equation}

\section{A relaxation framework for convexity}
\label{sec:relax}

In this section, we concentrate on the relaxation of the standard
convexity constraints for clarity of exposition. Most of the
propositions presented below can be extended to the the
generalizations of convexity presented in
Section~\ref{sec:generalization}. Let $X$ be a bounded convex domain
of $\Rsp^d$, and let $\H$ be the set of continuous convex functions on
$X$.  We define $L_k$ as the set of linear forms $\ell$ on $\C(K)$
which can be written as
\begin{equation}
\ell(g) = g\left(\sum_{i=1}^k \lambda_i x_i\right) - \left(\sum_{i=1}^k \lambda_i g(x_i)\right),
\label{eq:convform}
\end{equation}
where $x_1,\hdots,x_k$ are $k$ distinct points chosen in $X$, and
where $(\lambda_i)_{1\leq i\leq k}$ belongs to the $(k-1)$-dimensional
simplex $\Delta^{k-1}$. In other words, $\lambda_1,\hdots,\lambda_k$
are non-negative numbers whose sum is equal to one. Since we are only
considering continuous functions, $\H_{L_k}$ and $\H$
coincide as soon as $k\geq 2$.

We introduce in Definition~\ref{def:discretization} below a
discretization $M_\eps$ of the set of convexity constraints
$L_2$. Chon\'e and Le Meur proved in \cite{chone2001non} that the
union of the spaces of piecewise-linear convex functions on regular
grids of the square is \emph{not dense} in the space of convex
functions on this domain. This means that we need to be very careful
in order to apply the convexity constraints $M_\eps$ to
finite-dimensional spaces of functions. If one considers the space
$E_\delta$ of piecewise-linear functions on a triangulation of the
domain with edgelengths bounded by $\delta$, then $\H_{M_\eps} \cap
E_\delta = \H \cap E_\delta$ as soon as $\eps \ll \delta$. In this
case, one can fall in the pitfall identified by Chon\'e and Le Meur,
as illustrated in Figure~\ref{fig:chone}. Our goal in this section is
to show that it is possible to choose $\eps$ as a function of $\delta$
so that $\H_{M_\eps} \cap E_\delta$ becomes dense in $\H$ as $\delta$
goes to zero. Before stating our main theorem, we need to introduce
some definitions.

\begin{figure}
\centering
\includegraphics[width=5.7cm]{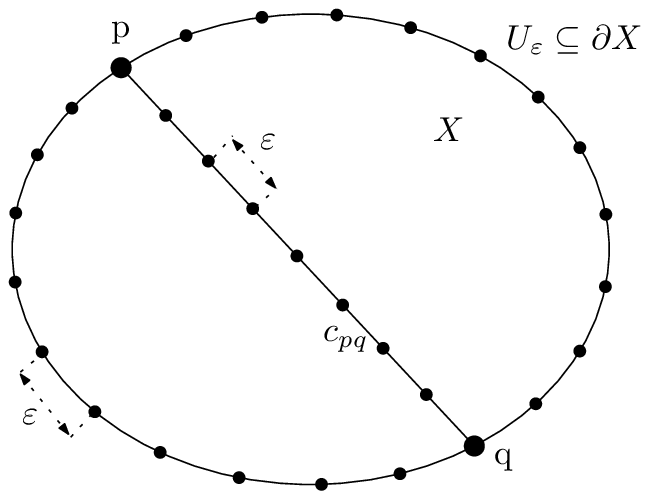}~~
\includegraphics[width=5.7cm]{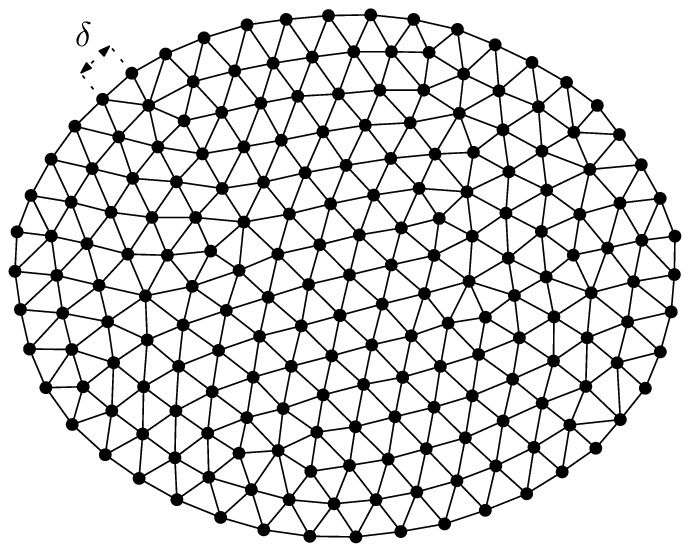}
\caption{\emph{(Left)} The discretized convexity constraints are
  enforced on the discrete segments $c_{pq}$ joining pairs of points of
  an $\eps$-sampling of the boundary $U_\eps \subseteq \partial
  X$. \emph{(Right)} A triangulation of the domain $X$, which can be
  used to define a finite-dimensional subspace $E_\delta$ of
  continuous piecewise-linear functions on $X$. Note that $\eps$ must be greater than $\delta$. The relation between the parameters $\delta$ and $\eps$ is
  studied in Theorems~\ref{th:main}--\ref{th:dh}. \label{fig:disc}}
\end{figure}

\begin{definition}[Discretized convexity constraints]
\label{def:discretization}
  Given any triple of points $(x,y,z)$ in $X$ such that $z$ belongs to
  $[x,y]$, we define a linear form $\ell_{xyz}$ by
\begin{equation} 
\ell_{xyz}(g) := g(z) - \frac{\nr{zy}}{\nr{xy}} g(x) - \frac{\nr{xz}}{\nr{xy}} g(x).
\end{equation}
Here and below, we set $\nr{xy} := \nr{x-y}$ to emphasize the notion
of distance between $x$ and $y$.  By convention, when we write
$\ell_{xyz}$, we implicitly assume that $z$ lies on the segment
$[x,y]$.  Consider a subset $U_\eps\subseteq \partial X$ such that for
every point $x$ in $\partial X$ there exists a point $x_\eps$ in
$U_\eps$ with $\nr{x - x_\eps} \leq \eps$ (see
Figure~\ref{fig:disc}). Given any pair of distinct points
$(p,q)$ in $U_\eps$, we let $c_{pq}$ be the discrete segment defined
by
$$c_{pq} := \left\{ p + \eps i
\frac{(q-p)}{\nr{q-p}}; i \in \Nsp,\, 0 \leq i \leq \nr{q-p}/\eps \right\}.$$
Finally, we define the following discretized 
set of constraints
\begin{equation}
M_\eps := \left\{ 
\ell_{xyz};~ x,y,z \in c_{pq} \hbox{ for some } p,q \in U_\eps \hbox{ and } z\in [x,y] \right\}.
\end{equation}
\end{definition}

\begin{definition}[Interpolation operator]
\label{def:interpolation}
  A \emph{(linear) interpolation operator} is a continuous linear map
  $\I_\delta$ from the space $\C(X)$ to a finite-dimensional linear
  subspace $E_\delta \subseteq \C(X)$, whose restriction to $E_\delta$
  is the identity map. We assume that the space $E_\delta$ contains
  the affine functions on $X$, and that the linear interpolation
  operator $I_\delta$ satisfies these properties:
\begin{align*}
\Lip(\I_\delta f) &\leq C_I \Lip f  \tag{L1} \label{eq:L1} \\
 \nr{f - \I_\delta f}_\infty &\leq \delta \Lip(f), 
\tag{L2} \label{eq:L2} \\
 \nr{f - \I_\delta f}_\infty &\leq \frac{1}{2} \delta^2 \Lip(\nabla f), 
\tag{L3} \label{eq:L3} 
\end{align*}
In practice, we consider families of linear interpolation operators
parameterized by $\delta$, and \emph{we assume that $C_I\geq 1$ is constant
  for the whole family.}
\end{definition}

\begin{example}
  Consider a triangulation of a polyhedral domain $X$, such that each
  triangle has diameter at most $\delta$. Define $E_\delta$ as the
  space of functions that are linear on the triangles of the mesh, and
  $\I_\delta(f)$ as the linear interpolation of $f$ on the mesh. Then
  $(E_\delta,I_\delta)$ is an interpolation operator and satisfies
  \eqref{eq:L1}--\eqref{eq:L3}. Other interpolation operators can be
  derived from higher-order finite elements, tensor-product splines,
  etc.
\end{example}

\begin{definition}[Superior limit of sets]
  The superior limit of a sequence of subsets $(A_n)$ of $\C(X)$ is
  defined by
$$ \overline{\lim}_{n\to \infty} A_n :=
\{ f \in \C(X);\exists f_n \in A_n, s.t. \lim_{n\to\infty} f_n = f \}.$$
\end{definition}

The following theorem shows that the non-density phenomenom identified
by Choné and Le Meur doesn't occur when $\eps$ is chosen large enough,
as a function of $\delta$. This theorem is a corollary of the more
quantitative Theorem~\ref{th:dh}. The remainder of this section is
devoted to the proof of Theorem~\ref{th:dh}.

\begin{theorem}
\label{th:main}
Let $X$ be a bounded convex domain $X$ and $(I_\delta)_{\delta > 0}$
be a family of linear interpolation operators. Let $f$ be a function
from $\Rsp_+$ to $\Rsp_+$ s.t.
$$\lim_{\delta \to 0} f(\delta) = 0
\qquad \lim_{\delta\to 0} \delta/f(\delta)^2 = 0.$$ We let
$\B_\Lip^\gamma$ denotes the set of $\gamma$-Lipschitz functions on
$X$. Then,
\begin{equation}
  \H\cap \B_\Lip^{\gamma/C_I} \subseteq \left[\overline{\lim}_{\delta \to 0} \H_{M_{f(\delta)}} \cap \B_\Lip^\gamma\right] \subseteq 
  \H \cap \B_\Lip^\gamma.
\end{equation}
\end{theorem}

\subsection{Relaxation of convexity constraints}

The first step needed to prove Theorem~\ref{th:main} is to show that
every function that belong to the space $\H_{M_\eps}$ is close to a
convex function on $X$ for the norm $\nr{.}_\infty$. This result
follows from an explicit upper bound on the distance between any
function in $\H_{M_\eps}$ and its convex envelope.

\begin{definition}[Convex envelope]
  Given a function $g$ on $X$, we define its \emph{convex envelope}
  $\bar{g}$ by the following formula:
\begin{equation}
\bar{g}(x) := \min \left\{ \sum_{i=1}^{d+1} \lambda_i g(x_i);~ x_i
\in X, \lambda \in \Delta^{d} \hbox{ and } \sum_i \lambda_i x_i =
x\right\},
\label{eq:ch}
\end{equation}
where $\Delta^d$ denotes the $d$-simplex, i.e. $\lambda \in
\Rsp_+^{d+1}$ and $\sum \lambda_i = 1$.  The function $\bar{g}$ is
convex and by construction, its graph lies below the graph of $g$.
\end{definition}

\begin{proposition}
\label{prop:relax}
For any function $g$ in the space $\H_{M_\eps}$, the distance between
$g$ and its convex envelope $\bar{g}$ is bounded by $\nr{g -
  \bar{g}}_\infty \leq \const(d) \Lip(g) \eps$.
\end{proposition}

This proposition follows from a more general result concerning a
certain type of relaxation of convexity constraints, which we call
$\alpha$-relaxation.

\begin{definition}[$\alpha$-Relaxation]
  Let $M,L$ be two sets of affine forms on the space $\C(X)$. The set
  $M$ is called an \emph{$\alpha$-relaxation} of $L$, where $\alpha$
  is a function from $\C(X)$ to $\Rsp_+ \cup \{+\infty\},$ if the
  following inequality holds:
\begin{equation}\forall \ell \in L, \forall g \in \C(K), ~\exists
\ell_g \in M,~ \abs{\ell(g) - \ell_g(g)} \leq \alpha(g).
\label{eq:alpharelax}
\end{equation}
\end{definition}

\begin{proposition}
\label{prop:convex}
  Consider an $\alpha$-relaxation $M$ of $L_2$. If $g$ lies in $\H_M$,
  the distance between $g$ and its convex envelope $\bar{g}$ is
  bounded by $\nr{g - \bar{g}}_\infty \leq d \alpha(g)$.
\end{proposition}
\begin{proof}
Let us show first that, assuming that $g$ is in $\H_{M}$, the
following inequality holds for any form $\ell$ in $L_{k}$:
\begin{equation}
\ell(g) \leq k \alpha(g).\label{eq:bell}
\end{equation}
For $k=2$, this follows at once from our hypothesis. Indeed, there
must exist a linear form $\ell_g$ in $M$ that satisfies
\eqref{eq:alpharelax}, so that $\ell(g) \leq \ell_g(g) + \alpha(g)$. Since
$g$ lies in $\H_M$, $\ell_g(g)$ is non-positive and we obtain
\eqref{eq:bell}. The case $k>2$ is proved by induction. Consider
$\lambda$ in the simplex $\Delta^{k-1}$ and points $x_1,\hdots,x_{k}$
in $X$. We assume $\lambda_1 < 1$ and we let $\mu_i =
\lambda_i/(1-\lambda_1)$ for any $i\geq 2$. The vector $\mu =
(\mu_2,\hdots,\mu_k)$ lies in $\Delta^{k-2}$, and therefore $y =
\sum_{i\geq 2} \mu_i x_i$ belongs to $X$. Applying the inductive
hypothesis \eqref{eq:bell} twice, we obtain:
\begin{align*}
&g\left(\lambda_1 x_1 + (1-\lambda_1) y\right) - \left(\lambda_1 g(x_1) +
  (1-\lambda_1) g\left(y\right)\right) \leq \alpha(g), \\
&g(y) - \left(\sum_{i=2}^k \mu_i g(x_i)\right) \leq (k-1) \alpha(g).
\end{align*}
The sum of the first inequality and $(1-\lambda_1)$ times the second
one gives \eqref{eq:bell}.  Now, consider the convex envelope
$\bar{g}$ of $g$.  Given any family of points $(x_i)$ and coefficients
$(\lambda_i)$ such that $\sum \lambda_i x_i = x$, we consider the form
$\ell(f) := f(x) - \sum_i \lambda_i f(x_i)$. Applying equation
\eqref{eq:bell} to $\ell$ gives
$$g(x) - d \alpha(g) \leq \sum \lambda_i g(x_i)$$
Taking the minimum over the $(x_i)$, $(\lambda_i)$ such that $\sum
\lambda_i x_i = x$, we obtain the desired inequality $\abs{g(x) -
  \bar{g}(x)}\leq d \alpha(g)$.
\end{proof}

In order to deduce Proposition~\ref{prop:relax} from
Proposition~\ref{prop:convex}, we should take $\alpha(g)$ proportional
to $\Lip(g)$. We use a technical lemma that gives an upper bound on
the difference between two linear forms corresponding to convexity
constraints in term of $\Lip$.

\begin{lemma}
Let $x,y,z$ and $x',y',z'$ be six points in $X$. Assume the following
\begin{itemize}
\item[(i)] $\max(\nr{x - x'}, \nr{y - y'}, \nr{z - z'}) \leq \eta$;
\item[(ii)] $z\in [x,y]$, $z'\in [x',y']$.
\end{itemize}
Then, $\abs{\ell_{xyz}(g) - \ell_{x'y'z'}(g)} \leq 6\eta\Lip(g)$.
\label{lem:tech}
\end{lemma}

\begin{proof}
We define $\lambda$ by the relation $z = \lambda x + (1-\lambda) y$,
and $\lambda'$ is defined similarly. We also define $\ell_i(g) := g(z)
- (\lambda' g(x) + (1-\lambda') g(y))$. Then,
$$ \abs{\ell_{xyz}(g) - \ell_{x'y'z'}(g)} \leq \abs{\ell_{x'y'z}(g) - \ell_i(g)} +
\abs{\ell_i(g) - \ell_{xyz}(g)} $$ The first term is easily bounded by $2
\eta \Lip(g)$, while the second term is bounded by $\abs{\lambda -
  \lambda'} \Lip(g) \nr{xy}$.
\begin{align*}
\abs{\lambda - \lambda'} &= \abs{\frac{\nr{zy}}{\nr{xy}} - \frac{\nr{z'y'}}{\nr{x'y'}}}  \\
&\leq \abs{\frac{\nr{zy} - \nr{z'y'}}{\nr{xy}}} + \abs{\frac{\nr{z'y'}}{\nr{x'y'}} \cdot \frac{\nr{x'y'} - \nr{xy}}{\nr{xy}}} \leq 4 \eta / \nr{xy}
\end{align*}
Overall, we get the desired  upper bound.\end{proof}

\begin{proof}[Proof of Proposition~\ref{prop:convex}]
  Our goal is to show that $M_\eps$ is an $\alpha$-relaxation of
  $L_2$. Consider three points $x,y,z$ in $X$ such that $z$ lies
  inside the segment $[x,y]$. The straight line $(x,y)$ intersects the
  boundary of $X$ in two points $a$ and $b$. By hypothesis, there
  exists two points $p$ and $q$ in $U_\eps$ such that the distances
  $\nr{a - p}$ and $\nr{b - q}$ are bounded by $\eps$. The maximum
  distance between the segments $[a,b]$ and $[p,q]$ is then also
  bounded by $\eps$ and the maximum distance between the segment
  $[a,b]$ and the finite set $c_{pq}$ by $2\eps$. This means that
  there exists three points $x_\eps,y_\eps$ and $z_\eps$ in $c_{pq}$
  such that $\max(\nr{x-x_\eps}, \nr{y-y_\eps}, \nr{y-y_\eps}) \leq
  2\eps$. Using Lemma~\ref{lem:tech}, we deduce that
  $\nr{\ell_{xyz}(g) - \ell_{x_\eps y_\eps z_\eps}(g)} \leq \alpha(g)
  := 12 \eps \Lip(g)$. This implies that $M_\eps$ is an
  $\alpha$-relaxation of $L_2$, and the statement follows from
  Proposition~\ref{prop:convex}.
\end{proof}

\subsection{Hausdorff approximation}
In this section, we use the estimation of the previous paragraph to
prove a quantitative version of Theorem~\ref{th:main}, using the
notion of directed Hausdorff distance.

\begin{definition}[Hausdorff distances]
  The directed or half-Hausdorff distance between two subsets $A, B$
  of a $\C(X)$ is  denoted $\hH(A\vert B)$:
\begin{equation}
\hH(A\vert B) = \min \left\{ r\geq 0;
\forall f\in A,~\exists g \in B,~\nr{f - g}_{\infty} \leq r \right\}.
\end{equation}
Note that this function is not symmetric in its argument.
\end{definition}

\begin{theorem}
\label{th:dh}
Let $X$ be a bounded convex domain of $\Rsp^d$ and $I_\delta:\C(X) \to
E_\delta$ be interpolation operator satisfying
\eqref{eq:L1}--\eqref{eq:L3}. We let $\B_\Lip^\gamma$ be the set
of $\gamma$-Lipschitz functions on $X$. Then, assuming $\gamma \geq 2
C_I \diam(X)$,
\begin{align}
\hH(\B_\Lip^\gamma \cap E_\delta \cap \H_{M_\eps} \,\vert\, \B_\Lip^\gamma \cap
\H) &\leq \const(d) \gamma \eps.\label{eq:h1}\\
  \hH(\B_\Lip^{\gamma/C_I} \cap \H \,\vert\, \B_\Lip^\gamma \cap E_\delta \cap
  \H_{M_{\eps}}) &\leq \const \frac{\gamma^2 \delta}{\eps^2} \diam(X),
\label{eq:h2}
\end{align}
Let $\B_{\C^{1,1}}^\kappa$ be the set of functions with
$\kappa$-Lipschitz gradients ($\kappa \geq 1$). Then,
\begin{equation}
  \hH(\B_{\C^{1,1}}^{\kappa} \cap \H \,\vert\, E_\delta \cap
  \H_{M_{\eps}}) \leq 
 \const \cdot \kappa^2\diam(X)^2 \frac{\delta^2}{\eps^2}
\label{eq:h3}
\end{equation}
\end{theorem}

\begin{remark}[Choice of the parameter $\eps$] The previous theorem
  has  implications on how to choose $\eps$ as a function of $\delta$
  in order to obtain theoretical convergence results. In practice, the
  estimations given by the items (i) and (ii) below seem to be rather
  pessimistic, and in applications we always choose $\eps$ to be a
  small constant times $\delta$.

  \begin{itemize}
  \item[(i)] If one chooses $\eps = f(\delta)$, where $f$ is a
    function from $\Rsp_+$ to $\Rsp_+$ such that $\lim_{\delta \to 0}
    f(\delta) = 0$ and $\lim_{\delta\to 0} \delta/f(\delta)^2 = 0$,
    the upper bounds in Equations~\eqref{eq:h1}--\eqref{eq:h2}
    converges to zero when $\delta$ does. This implies the convergence
    result stated in Theorem~\ref{th:main}.
  \item[(ii)] We can choose $\eps$ so as to equate the two upper bound
    in Equations~\eqref{eq:h1}--\eqref{eq:h2}, i.e. $\eps \simeq
    \delta^{1/3}$. This suggests that the best rate of convergence in
    Hausdorff distance that one can expect from this analysis, in
    order to recover all convex functions $\H \cap \B_\Lip^\gamma$, is in $\BigO(\delta^{1/3})$.
  \item[(iii)] On the other hand, Equation~\eqref{eq:h3} shows that
    convex and $\Class^{1,1}$ functions are easier to approximate by
    discrete convex functions. In particular, if $f$ is merely a
    superlinear function, i.e. $\lim_{\delta\to 0} f(\delta) = 0$ and
    $\lim_{\delta\to 0} \delta/f(\delta) = 0$, then, with $\eps =
    f(\delta)$, the upper bounds of both Equation~\eqref{eq:h1} and
    \eqref{eq:h3} converge to zero.
\end{itemize}
\end{remark}

The following easy lemma shows that the space $\H_{M_\eps} \cap
E_\delta$ has non-empty interior for the topology induced by the
finite-dimensional vector space $E_\delta$ as soon as $\delta < \eps$.
This very simple fact is the key to the proof of Theorem~\ref{th:dh}.

\begin{lemma}
\label{lemma:convnonempty}
Consider the function $s(x) := \nr{x-x_0}^2$ on $X$, where $x_0$ is a
point in $X$, and the interpolating function $s_\delta := \I_\delta
s$. Then, $$\max_{\ell \in M_\eps} \ell(s_\delta) \leq \delta^2 -
\eps^2.$$
\end{lemma}

\begin{proof}
Consider three points $x < z < y$ on the real line, such that $\abs{x
  - z} \geq \eps$ and $\abs{y - z} \geq \eps$ and $z = \lambda x +
(1-\lambda) y$. Then,
\begin{align*}
 z^2 - \lambda x^2 - (1- \lambda) y^2
&= z^2 - \lambda (z + (x-z))^2  - (1-\lambda) (z + (y-z))^2 \\
&= - [\lambda (x-z)^2 + (1-\lambda) (y-z)^2] \leq -\eps^2
\end{align*}
Since the gradient of $s$ is $2$-Lipschitz, using \eqref{eq:L3} we get
$\nr{s - s_\delta}_\infty \leq \delta^2$. Combining with the previous
inequality, this implies $\ell(s_\delta) \leq \delta^2 - \eps^2$ for
every linear form $\ell$ in $M_\eps$.
\end{proof}

\begin{proof}[Proof of Theorem~\ref{th:dh}]
  Let $g$ be a function in the intersection $\H_{M_\eps}
  \cap \B_\Lip^\gamma$.  Then, Proposition~\ref{prop:convex} implies
  that its convex envelop $\bar{g}$ satisfies $\nr{g - \bar{g}}_\infty
  \leq \const(d) \Lip(g)\eps$.  The Lipschitz constant of a function
  is not increased by taking its convex envelope, and thus $\bar{g}$
  belongs to $\H \cap \B_\Lip^\gamma$.  This implies the upper bound
  given in Equation~\eqref{eq:h1}.

  On the other hand, given a convex function $f$ in $\H\cap
  \B_\Lip^{\gamma/C_I}$, we consider the function $g := \I_\delta f$
  defined by the interpolation operator. By property \eqref{eq:L1} the
  function $g$ belongs to $\B_\Lip^\gamma$, and by property \eqref{eq:L2}
  one has for any linear form $\ell$ in $L_2$,
\begin{align*}
 \ell(g) &= g(\lambda x + (1-\lambda) y)  - (\lambda g(x) + (1-\lambda) g(y)) \\
&\leq f(\lambda x + (1-\lambda) y) - (\lambda f(x)  + (1-\lambda) f(y)) + 2\delta \gamma 
\leq 2\delta\gamma 
\end{align*}
For $\eta < 1$, we let $g_\eta := (1-\eta) g + \eta
s_\delta$. Assuming $\delta \leq \eps/2$, the previous inequality
implies that for any linear form $\ell$ in $M_\eps$,
\begin{equation*}
\ell(g_\eta) \leq (1-\eta) 2\delta \gamma + \eta(\delta^2 - \eps^2)
\leq 2\delta \gamma - \eta\eps^2/2
\end{equation*}
Consequently, assuming $4\delta\gamma \leq \eta \eps^2$ the inequality
$\ell(g_\eta) \leq 0$ holds for any linear form $\ell$ in $M_\eps$,
and $g$ belongs to $\H_{M_\eps}$. Moreover, using the fact that $\Lip$
is a semi-norm, we have
\begin{align}
\Lip(g_\eta) &\leq (1-\eta) \Lip(g) + \eta \Lip(s_\delta)  \notag\\
&\leq (1-\eta) \gamma + 2C_I\eta \diam(X)
\label{eq:blip}
\end{align}
In the second inequality, we used property \eqref{eq:L1} and $\Lip(s)
\leq 2 \diam(X)$. By Equation~\ref{eq:blip}, the function $g_\eta$
belongs to $\B_\Lip^\gamma$ provided that $\gamma \geq 2 C_I
\diam(X)$.  From now on, we fix $\eta = 4\delta\gamma/\eps^2$, and let
$h = g_\eta$, which by the discussion above belongs to $\H_{M_\eps}\cap
\B_\Lip^\gamma$.  The distance between $f$ and $h$ is bounded by
$$ \nr{f - h}_\infty \leq \nr{f-g}_\infty +
\eta (\nr{g}_\infty + \nr{s_\delta}_\infty).$$ Since the space
$E_\delta$ contains constant functions, we can assume that
$\nr{g}_\infty$ is bounded by $\Lip(g) \diam(X) \leq \gamma\diam(X)$. Assuming $\eps\leq\diam(X)$,
\begin{align*}
  \nr{f - h}_\infty
  \leq \gamma \delta\left[1 + \frac{4}{\eps^2} \left(\gamma \diam(X) + \diam(X)^2\right)\right] 
  \leq 10 \frac{\gamma^2 \delta}{\eps^2} \diam(X),
\end{align*}
thus implying Equation~\eqref{eq:h2}

The proof of Equation~\eqref{eq:h3} is very similar. Given a convex
function $f$ in the intersection $\H\cap \B_{\Class^{1,1}}^{\kappa}$,
we consider the function $g := \I_\delta f$ defined by the
interpolation operator. Using \eqref{eq:L3} one has for any linear
form $\ell$ in $L_2$,
\begin{align*}
 \ell(g) &= g(\lambda x + (1-\lambda) y)  - (\lambda g(x) + (1-\lambda) g(y)) \\
&\leq f(\lambda x + (1-\lambda) y) - (\lambda f(x)  + (1-\lambda) f(y)) +  \delta^2 \kappa
\leq \delta^2 \kappa
\end{align*}
For $\eta < 1$, we set $g_\eta := (1-\eta) g + \eta
s_\delta$. Assuming $\delta \leq \eps/2$, the previous inequality
implies that for any linear form $\ell$ in $M_\eps$,
\begin{equation*}
\ell(g_\eta) \leq (1-\eta) \delta^2 \kappa + \eta(\delta^2 - \eps^2)
\leq \delta^2\kappa - \eta\eps^2/2
\end{equation*}
Hence, the function $h:=g_\eta$ belongs to $\H_{M_\eps}$, where $\eta
:= 2\kappa \delta^2/\eps^2$. Using the fact that $E_\delta$ contains
affine function, we can assume $g(x_0) = 0$, $\nabla g(x_0) = 0$ for
some point $x_0$ in $X$, so that $\nr{g}_\infty \leq \kappa
\diam(X)^2$. Combining this with property \eqref{eq:L3}, we get
the following upper bound, which implies Equation~\ref{eq:h3}:
\begin{align*}
\nr{f-h}_{\infty} &\leq \nr{f-g}_\infty + \eta(\nr{g}_\infty + \nr{s_\delta}_\infty) \\
&\leq \const \cdot \kappa\delta^2
\left[1 + \frac{\kappa \diam(X)^2 + \diam(X)^2}{\eps^2}\right].
\qedhere
\end{align*}
\end{proof}

\section{Generalization to convexity-like constraints}
\label{sec:generalization}
In \S\ref{subsec:support} we show how to extend the relaxation of
convexity constraints presented above to the constraints arising in
the definition of the space of support function of convex bodies. We
show in \S\ref{subsec:cconv} that both type of constraints fit in the
general setting of $c$-convexity constraints, where $c$ satisfy the
so-called non-negative cross-curvature condition.

\subsection{Support functions}
\label{subsec:support}
A classical theorem of convex geometry, stated as Theorem~1.7.1 in
\cite{schneider} for instance, asserts that any compact convex body in
$\Rsp^d$ is uniquely determined by its \emph{support function}. The
support of a convex body $K$ is defined by the following formula
$$\underline{\hh}_K: x\in \Rsp^d\mapsto \max_{p\in K} \sca{x}{p}.$$
This function is is positively $1$-homogeneous and is therefore
completely determined by its restriction $\hh_K$ on the unit
sphere. We consider the space $\H^s \subseteq \C(\Sph^{d-1})$ of
support functions of compact convex sets. This space coincides with
the space of bounded functions on the sphere whose $1$-homogeneous
extensions to the whole space $\Rsp^d$ are convex.

\begin{lemma}
A bounded function $g$ on the unit sphere is the support function of a
bounded convex set if and only if for every $x_1,\hdots,x_k$ in the
sphere, and every $(\lambda_1,\hdots,\lambda_k) \in \Delta^{k-1}$,
\begin{equation}
\nr{x} g\left(\frac{x}{\nr{x}}\right) \leq \sum_i \lambda_i g(x_i),~\hbox{ where } x:= \sum_i \lambda_i x_i.
\end{equation}
Moreover, $g$ is the support function of a convex set if it satisfies
the inequalities for $k=2$ only.
\end{lemma}

Following this lemma, we define $L_k^s$ as the space of all linear
forms that can be written as
\begin{equation}
\ell(g) := \nr{\sum_i \lambda_i x_i} g\left(\frac{\sum_i \lambda_i x_i}{\nr{\sum_i \lambda_i x_i}}\right) - \sum_i \lambda_i g(x_i),
\label{eq:suppcons}
\end{equation}
where $x_1,\hdots,x_k$ are points on the sphere $\Sph^{d-1}$, and
$(\lambda_1,\hdots,\lambda_k)$ lies in $\Delta^{k-1}$. With this
notation at hand, we have another characterization of the space of
support functions: $\H^s$ coincides with the spaces $\H_{L^s_k}$ for
any $k\geq 2$.

\subsubsection{Discretization of the constraints}
The discretization of the set $L^s_2$ of constraints satisfied by
support functions follows closely the discretization of the convexity
constraints described earlier. Consider three points $x$, $y$ and $z$
such that $x$ and $y$ are not antipodal and such that $z$ belongs to
the minimizing geodesic between $x$ and $y$. We let $z'$ be the radial
projection of $z$ on the extrinsic segment $[xy]$, i.e.  such that
$z'/\nr{z'} = z$. Finally, we let $\lambda = \nr{zy}/\nr{xy}$ and
define:
$$\ell_{xyz}(g) := \nr{z'}g(z) - \lambda g(x) - (1-\lambda) g(y).$$ 
As before, we discard the constraint $\ell_{xyz}$ if $z$ does not lie
on the minimizing geodesic arc between $x$ and $y$. Let $U_\eps$ be a
subset of the sphere that satisfies the sampling condition
\begin{equation}
\forall u \in \Sph^{d-1},~\exists (\sigma,v) \in \{\pm 1\} \times
U_\eps,~\hbox{s.t. } \nr{u - \sigma v} \leq \eps.
\label{eq:sampling}
\end{equation}
Then, for every vector $u$ in $U_\eps$ we construct an $\eps$-sampling
$c_u$ of the great circle orthogonal to $u$ that is also
$\frac{\eps}{2}$-sparse, i.e.  $\, \nr{x-y} \geq \frac \varepsilon 2$
for any pair of distinct points $x,y$ in $c_u$.  The space of
constraints we consider is the following:
\begin{equation*}
 M_\eps^s = \left\{ \ell_{xyz};~x, y, z \in c_u \hbox{ for some } u \in U_\eps \right\}.
\end{equation*}
The proof of the following statement follows the proof of
Proposition~\ref{prop:relax}, and even turns out to be slightly
simpler as one does not need to take care of the boundary of the
domain.

\begin{proposition}
\label{prop:relaxsupp}
For any function $h$ in the space $\H_{M_\eps^s}$, there exists a
bounded convex set $K$ such that $\nr{h - h_K}_\infty \leq
\const(d) \Lip(g) \eps$.
\end{proposition}

It is possible to define a notion of interpolation operator on the
sphere as in Definition~\ref{def:interpolation}, and to obtain
Hausdorff approximation results similar to those presented in
Theorem~\ref{th:main}. The statement and proofs of the theorem being
very similar, we do not repeat them. However, we show that the
indicator function of the unit ball, i.e. the constant function equal
to one, belongs to the interior of the set $\H_{M_\eps^s}$. This is
the analogous of Lemma~\ref{lemma:convnonempty}, which was the crucial
point of the proof of convergence for the usual convexity.
\begin{lemma}
With $s(x) :=1$, one has $\max_{\ell \in M_\eps^s} \ell(s) \leq -  \const \cdot \eps^2$.
\label{lemma:suppnonempty}
\end{lemma}
\begin{proof}
For every $\ell$ in $M_\eps^s$, there exists three (distinct) points
$x,y,z$ in $c_u$ for some $u$ in $U_\eps$. Let $z'$ denote the radial
projection of $z$ on the segment $[x,y]$. Then, $\ell_{xyz}(s) =
\nr{z'} - 1$. By construction, $\nr{x-z}$ and $\nr{y-z}$ are at least
$\eps/2$, and therefore $\nr{z'} \leq 1 - \const\cdot \eps^2$ thus proving the lemma.
\end{proof}

\subsubsection{Support function as $c$-convex functions} 
Oliker \cite{oliker2007embedding} and
Bertrand \cite{bertrand2010prescription} introduced another
characterization of support functions of convex sets, inspired by
optimal transportation theory. They show that logarithm of support
functions coincide with $c$-convex functions on the sphere for the
cost function $c(x,y) = -\log(\max(\sca{x}{y},0))$ (see
\S\ref{subsec:cconv} for a definition of $c$-convexity):
\begin{lemma}
  The $1$-homogeneous extension of a bounded positive function $h$ on
  $\Sph^{d-1}$ is convex if and only if the function $\phi := \log(h)$
  can be written as
$$ \phi(x) = \sup_{y \in \Sph^{d-1}} - \psi(y) - c(x,y) $$
where $c(x,y) = -\log(\max(\sca{x}{y},0))$ and $\psi: \Sph^{d-1}\to\Rsp$.
\end{lemma}
\begin{proof} We show only the direct implication, the reverse implication can be found in
\cite{bertrand2010prescription}. By assumption $h = h_K$, where $K$ is
a bounded convex set that contains the origin in its interior, and let
$\rho_K$ be the radial function of $K$ i.e. $\rho_K(y) := \max \{ r;\,
r y \in K\}$. Then,
$$ h_K(x) = \max_{p \in K}\sca{x}{p} = \max_{y \in \Sph^{d-1}}
  \rho_K(y) \sca{x}{y}$$
Since $h_K > 0$, the maximum in the right-hand side is attained for a point $y$ such that
$\sca{x}{y} > 0$. Taking the logarithm of this expression, we get:
\begin{equation*} \phi(x) = \max_{y \in \Sph^{d-1}} \log(\rho_K(y)) - c(x,y),
\end{equation*}
thus concluding the proof of the direct implication.
\end{proof}

\subsection{ $c$-Convex functions}
\label{subsec:cconv}
In this paragraph, we show how the discretizations of the spaces of
convex and support functions presented above can be extended to
$c$-convex functions. This extension is motivated by a generalization
of the principal-agent problem proposed by Figalli, Kim and McCann
\cite{figalli2011multidimensional}. Thanks to the similarity between
the standard convexity constraints and those arising in their setting,
one could hope to perform numerical computations using the same type
of discretization as those presented in Section~\ref{sec:relax}.

The authors of \cite{figalli2011multidimensional} prove that the set
of $c$-convex functions is convex if and only if $c$ satisfy the
so-called non-negative cross-curvature condition. Under the same
assumption, we identify the linear inequalities that define this
convex set of functions. Note that the numerical implementation of
this section is left for future work.

Given a cost function $c:X\times Y \to \Rsp$, where $X$ and $Y$ are
two open and bounded subsets of $\Rsp^d$, the $c$-transform and
$c^*$-transform of lower semi-continuous functions $\phi: X\to\Rsp$
and $\psi: Y\to\Rsp$ are defined by
\begin{align*}
\phi^{c^*}(y) &:= \sup_{x\in Y} - c(x,y) - \phi(x), \\
\psi^{c}(x) &:= \sup_{y \in Y} - c(x,y) - \psi(y). 
\end{align*}
A function is called $c$-convex if it is the $c^*$-transform of a
lower semi-continuous function $\psi: Y\to\Rsp$. The space of
$c$-convex functions on $X$ is denoted $\H_c$.  We will need the
following usual assumptions on the cost function $c$:
\begin{enumerate}
\item[(A0)] $c \in \C^4(\bar{X}\times \bar{Y})$, where $X, Y \subseteq
  \Rsp^n$ are bounded open domains.
\item[(A1)] For every point $y_0$ in $Y$ and $x_0$ in $X$, the maps
\begin{align*} 
x \in \bar{X} &\mapsto -\nabla_y c(x,y_0)\\
y \in \bar{Y} &\mapsto -\nabla_x c(x_0,y)
\end{align*}
are diffeomorphisms onto their range  (bi-twist).
\item[(A2)] For every point $y_0$ in $Y$ and $x_0$ in $X$, the sets
  $X_{y_0} := -\nabla_y c(X,y_0)$ and $Y_{x_0} := -\nabla_x c(x_0,Y)$
  are convex (bi-convexity).
\end{enumerate}
These conditions allow one to define the $c$-exponential map. Given
a point $y_0$ in the space $Y$, the $c$-exponential map $\exp_{y_0}^c:
X_{y_0} \to X$ is defined as the inverse of the map $-\nabla_y c(.,
y_0)$, i.e. it is the unique solution of
\begin{equation}
\exp_{y_0}^c (-\nabla_y c(x, y_0)) = x. \label{eq:exp}
\end{equation}
The following formulation of the non-negative cross-curvature
condition is slightly non-standard, but it agrees to the usual
formulation for smooth costs under conditions (A0)--(A2), thanks to Lemma~4.3 in
\cite{figalli2011multidimensional}.
\begin{enumerate}
\item[(A3)] For every pair of points
  $(y_0,y)$ in $Y$ the following map is convex:
\begin{equation}
v \in X_{y_0} \mapsto c(\exp^c_{y_0} v, y_0) - c(\exp^c_{y_0} v, y).
\label{eq:nncc}
\end{equation}
\end{enumerate}
The main theorem of \cite{figalli2011multidimensional} gives a
necessary and sufficient condition for the space $\H_c$ of $c$-convex
functions to be convex.
\begin{theorem}
Assuming (A0)--(A2), the space of $c$-convex functions $\H_c$ is
itself convex if and only if $c$ satisfies (A3).
\end{theorem}
The proof that (A0)--(A3) implies the convexity of $\H_c$ given in
\cite{figalli2011multidimensional} is direct but non-constructive, as
the authors show that the average of two functions $\phi_0$ and
$\phi_1$ in $\H_c$ also belongs to $\H_c$. The following proposition
provides a set of linear inequality constraints that are both
necessary and sufficient for a function to be $c$-convex.

\begin{proposition}
Assuming the cost function satisfies (A0)--(A3), a function $\phi:
X\to\Rsp$ is $c$-convex if and only if it satisfies the following
constraints:
\begin{itemize}
\item[(i)] for every $y$ in $Y$, the map $\phi_y: v \in X_y
  \mapsto \phi(\exp^c_y v) + c(\exp^c_y v, y)$ is convex.
\item[(ii)] for every $x$ in $X$, the subdifferential $\partial \phi(x)$ is included in $Y_x$.
\end{itemize}
\end{proposition}

Note that the first set of constraints (i) can be discretized in an
analogous way to the previous sections. On the other hand, the second
constraint concerns the subdifferential of $\phi$ in the sense of
semiconvex functions. It is not obvious how to handle this constraint
numerically, except in the trivial case where $Y_x$ coincides with the
whole space $\Rsp^d$ for any $x$ in $X$.

\begin{proof} Suppose first that  $\phi$ is $c$-convex. Then, 
there exists a function $\psi$ such that $\phi(x) = \psi^{c*}$ and
one has
$$\phi_y(v) = \sup_{z} [- \psi(z) - c(\exp^c_y v, z)] + c(\exp^c_y v,
y).$$ Equation~\ref{eq:nncc} implies that $\phi_y$ is convex as a maximum
of convex functions.

Conversely, suppose that a map $\phi:X\to\Rsp$ is such that the maps
$\phi_y$ are convex for any point $y$ in $Y$, and let us show that
$\phi$ is $c$-convex. Using the definition of $\phi_y$, and the
definition of the $c$-exponential \eqref{eq:exp}, one has
$$\phi(x) = \phi_y(-\nabla_y c(x,y)) - c(x,y)$$ for any pair of points
$(x,y)$ in $X\times Y$. This formula and the convexity of $\phi_y$
imply in particular that the map $\phi$ is semiconvex. Consequently,
for every point $x$ in $X$, the subdifferential $\partial \phi(x)$ is
non-empty, and there must exist a point $y$ in $Y$ such that $v :=
-\nabla_x c(x,y)$ belongs to $\partial \phi(x)$. Hence, $x$ is a
critical point of the map $\phi - c(.,y)$, and therefore $v$ is a
critical point of $\phi_y$. By convexity, $v$ is also a global minimum
of $\phi_y$, i.e. for every $w$ in $\bar{X}_y$,
$$\phi(\exp_y^c w) + c(\exp_y^c w,y) \geq \phi(x) + c(x,y).$$ Letting
$x' = \exp_y^c w$, we get $\phi(x') \geq \phi(x) + c(x,y) - c(x',y)$.
 The function $\phi(x) + c(x,y) - c(.,y)$ is thus supporting $\phi$
at $x$. Since $\phi$ admits such a supporting function at every point
$x$ in $X$, it is a $c$-convex function.
\end{proof}

\begin{comment}
\quentin{A way to discretize these constraints:}

This proposition leads us to define the space $L_k^c$, that consists
 of linear forms on the space $\C(X)$ that can be written as
 $$ \ell(g) = g\left(\sum_{i=1}^k \lambda_i g(\exp_y^c v_i)\right) -
 \left(\sum_{i=1}^k \lambda_i g(\exp_y^c v_i)\right), $$ where $y$ is
 chosen in $Y$, $v_1,\hdots,v_k$ are $k$ vectors in $X_y$, and
 $\lambda$ belongs to the $k$-simplex $\Delta^k$. As before, the
 spaces $\H_{L_k^c}$ coincide with $\H_c$ as soon as $k\geq 2$.

 One can therefore try to 
\end{comment}

\section{Numerical implementation}
In this section, we give some details on how to apply the relaxed
convexity constraints presented in Section \ref{sec:relax} to the
numerical solution of problems of calculus of variation with (usual)
convexity constraints. Our goal is to minimize a convex functional
$\mc{F}$ over the set of convex functions $\H$. Our algorithm assumes
that $\mc{F}$ is easily proximable (see Definition~\ref{def:prox}).
For any convex set $K$, we denote $\ii_K$ the convex indicator
function of $K$, i.e. the function that vanishes on $K$ and take value
$+\infty$ outside of $K$.  The constrained minimization problem can
then be reformulated as
\begin{equation}
\min_{g \in \C(X)} \mc{F}(g) + \ii_{\H}(g),
\label{eq:P}
\end{equation} 
The method that we present in this paragraph can be applied with minor
modifications to support functions. Its extension to the other types
of convexity constraints presented in
Section~\ref{sec:generalization}, will be the object of future
work.

\subsection{Discrete formulation}
We are given a finite-dimensional subspace $E$ of $\C(X)$, and a
linear parameterization $\P: \Rsp^N \to E$ of this space. This
subspace $E$ and its parameterization play a similar role to the
interpolation operator in the theoretical section.  For instance, we
can let $E$ be the space of piecewise-linear functions on a
triangulation of $X$, and $\P$ be the parameterization of this space
by the values of the function at the vertices of the triangulation.
For every point $x$ in $X$, this parametrization induces a linear
evaluation map $\P_x: \Rsp^N\to\Rsp$, defined by $\P_x \xi := (\P
\xi)(x)$. By convention, if $x$ does not lie in $X$, then $\P_x \xi =
+\infty$.  The convexity constraints in \eqref{eq:P} are discretized
using Definition~\ref{def:discretization}:
\begin{equation}
\min_{\xi \in \Rsp^N} \mc{F}(\P\xi) + \ii_{\H_{M_\eps}}(\P\xi). \label{eq:relprob}
\end{equation}
We now show how to rewrite the indicator function of the discretized
convexity constraints $\H_{M_\eps}$ as a sum of indicator
functions. This allows us to exploit this particular structure to
deduce an efficient algorithm.

Let $U_\eps \subseteq \partial X$ be a finite subset such that every
point of $\partial X$ is at distance at most $\eps$ from a point of
$U_\eps$. Given a pair of points $p\neq q$ in $U_\eps$, we
consider the discrete segment $c_{pq} := \left\{ p + \eps i
  (q-p)/\nr{q-p}; i \in \Nsp,\, 0 \leq i \leq \nr{q-p}/\eps \right\}.$
These geometric constructions are illustrated in
Figure~\ref{fig:disc}. The evaluation of a function $\P \xi$ on a
discrete segment $c_{pq}$ is a vector indexed by $\Nsp$, which takes
finite values only for indices in $\{0,\hdots,\abs{c_{pq}}-1\}$:
$$\P_{pq} \xi = \left(\P \xi\left(p + \eps i \frac{(q-p)}{\nr{q-p}}\right)\right)_{ i \in \Nsp}
$$
Define $\H_{1}$ as the cone of vectors $(f_i)_{i\in \Nsp}$ that
satisfy the discrete convexity conditions $f_i \leq \frac{1}{2}
(f_{i-1} + f_{i+1})$ for $i\geq 1$.  The relaxed problem
\eqref{eq:relprob} is then equivalent to the following minimization
problem:
\begin{equation}
\label{eqdiscrete}
\min_{\xi \in \Rsp^N} \mc{F}(\P\xi) + \sum_{\substack{(p,q) \in U_\eps^2\\
p\neq q}}
\ii_{\H_{1}}(\P_{pq} \xi).
\end{equation}

\begin{remark}[Number of constraints] In numerical applications, we
  set $\eps = c\delta$, where $c$ is a small constant, usually in the
  range $(1,3]$. For a fixed convex domain $X$ of $\Rsp^2$, there are
  $\BigO(1/\eps)$ constraints per discrete segment and
  $\BigO(1/\eps^2)$ such discrete segments. The total number of
  constraints is therefore $C := \BigO(1/\eps^3) = \BigO(1/\delta^3)$.
  Moreover, a triangulation of $X$ with maximum edgelength $\delta$
  has at least $N = \BigO(1/\delta^2)$ points. This implies that the
  dependence of the number of constraints as a function of the number
  of points is given by $C = \BigO(N^{3/2})$. This is slightly lower
  than the exponent $\BigO(N^{1.8})$ found in
  \cite{carlier2001numerical}. Moreover, as shown below, the structure
  constraints of Equation~\eqref{eqdiscrete} is favorable for
  optimization.
\end{remark}

\subsection{Proximal methods} \label{prox}

When the functional $\mc{F}$ is linear \eqref{eqdiscrete} is a
standard linear programming problem. Similarly, when $\mc{F}$ is quadratic and
convex, this problem is a quadratic programming problem with linear
constraints. Below, we show how to exploit the $1$D structure of the
constraints so as to propose an efficient and easy to implement
algorithm based on a proximal algorithm. This algorithm allows to
perform the optimisation when $\mc{F}$ is a more general function. The
version of the algorithm that we describe below is able to handle
functions that are \emph{easily proximable} (see
Definition~\ref{def:prox}). Note that it it would also be possible to
handle functions $\mc{F}$ whose gradient is Lipschitz using the
generalized forward-backward splitting algorithm of
\cite{raguet2013generalized}.

\begin{definition}[Proximal operator]
\label{def:prox} The proximal operator associated
  to a convex function $f: \Rsp^N \to \Rsp$ is defined as follows:
\begin{equation}\prox_\gamma f (y) = \arg\min_{x \in \Rsp^N} f(x) + \frac 1 \gamma \nr{x-y}^2.
\end{equation}
The function is called \emph{easily proximable} if there exists an
efficient algorithme able to computes its proximal operator. For
instance, when $f$ is the indicator function $i_K$ of a convex set,
$\prox_\gamma f$ coincides with the projection operator on $K$,
regardless of the value of $\gamma$.
\end{definition}

 The simultaneous-direction method
of multipliers (SDMM) algorithm is designed to solve convex
optimization problems of the following type :
$$\min_{x\in \Rsp^N} g_1(L_1x) + \dots g_m(L_mx)$$ where the
$(L_i)_{1\leq i \leq m}$ are matrices of dimensions $N_1\times N,
\dots, N_m \times N$ and the function $(g_i)_{1\leq i \leq m}$ are
convex and easily proximable. Moreover, it assumes that the matrix $Q
:= \sum_{i=1}^m L_i^T L_i$ is invertible, where $L_i^T$ stands for the
transpose of the matrix $L_i$. A summary of the SDMM algorithm is
given in Algorithm~\ref{algo:SDMM}. More details, and variants of this
algorithm can be found in \cite{bauschke2011fixed}. Note that when
applied to \eqref{eqdiscrete}, every iteration of the outer loop of
the SDMM algorithm involves the computation of several projection on
the cone of $1$D discrete functions $\H_1$. These projection can be
computed independently, thus allowing an easy parallelization of the
optimization.

\begin{algorithm}[t]
\caption{Simultaneous-direction method of multipliers (SDMM)}
\label{algo:SDMM}
\begin{description}
\item[Input] $\gamma >0$ 
\item[Initialization] $(y_{1,0},\,z_{1,0}) \in \Rsp^{2N_1} , \dots, (y_{m,0},\,z_{m,0}) \in \Rsp^{2N_m}$
\item[For] $n=0,1,\dots$\\
      $x_n = Q^{-1} \sum_{i=1}^m L_i^T(y_{i,n}-z_{i,n})$
\begin{description} 
\item[For] $i=1,\dots,m$ \\
 $s_{i,n}=L_i x_n$\\
$y_{i,n+1}=\prox_{\gamma}g_i(s_{i,n}+z_{i,n})$\\
$z_{i,n+1}=z_{i,n} + s_{i,n} - y_{i,n+1}$\\
\end{description}
\end{description}
\end{algorithm}

\subsection{Hinge algorithm}
In the algorithm above, we need to compute the $\ell^2$ projection of
a vector $(f_i)$ on the cone of discrete $1$D convex functions
$\H_{1}$. In practice, $(f_i)$ is supported on a finite set
$\{0,\hdots, n\}$, and one needs to compute the $\ell^2$ projection of
this vector onto the convex cone
$$ \H_1^n = \{ g: \{0,\hdots,n\} \to \Rsp;~\forall i \in \{
1,\hdots, n-1 \}, 2 g_i \leq g_{i-1} + g_{i+1} \}.$$ This problem is
classical, and several efficient algorithm have been proposed to solve
it. Since the number of conic constraints is lower than the dimension
of the ambient space ($n+1$), the number of extreme rays of the
polyhedral cone $\H_1^n$ is bounded by $n$. In this case coincide with
the edges of the cone. Moreover, as noted by Meyer
\cite{meyer2010algorithm}, these extreme rays can be computed
explicitly. This remark allows one to parameterize the cone $\H_1^n$
by the space $\Rsp\times \Rsp_+^{n}$, thus recasting the projection
onto $\H_1^n$ into a much simpler non-negative least squares
problem. To solve this problem, we use the simple and efficient exact
active set algorithm proposed by Meyer\,\cite{meyer2010algorithm}. In
our implementation, we reuse the active set from one proximal
computation to the next one. This improves the computation time by up
to an order of magnitude.

\section{Application I : Denoising \label{denoising-section}}
Our first numerical application focuses on the $L^2$ projection onto
the set of convex functions on a convex domain. We illustrate the
efficiency of our relaxed approach in the context of denoising.  Let
$u^*$ be a convex function on a domain $X$ in $\Rsp^d$. We approximate
this function by a piecewise linear function on a mesh, and the values
of the function at the node of the mesh are additively perturbed by
Gaussian noise: $u_0(p) = u^*(p) + c {\mathcal N}(0,1)$, where
${\mathcal N}(0,1)$ stands for the standard normal distribution and
$c$ is a small constant. Our goal is then to solve the following
projection problem in order to estimate the original function $u^*$:

$$\min_{u \in \H} \nr{u - u_0}_{L^2(X)}.$$

\begin{figure}
\centering
\begin{tabular}{r r }
\includegraphics[height=\widthh cm]{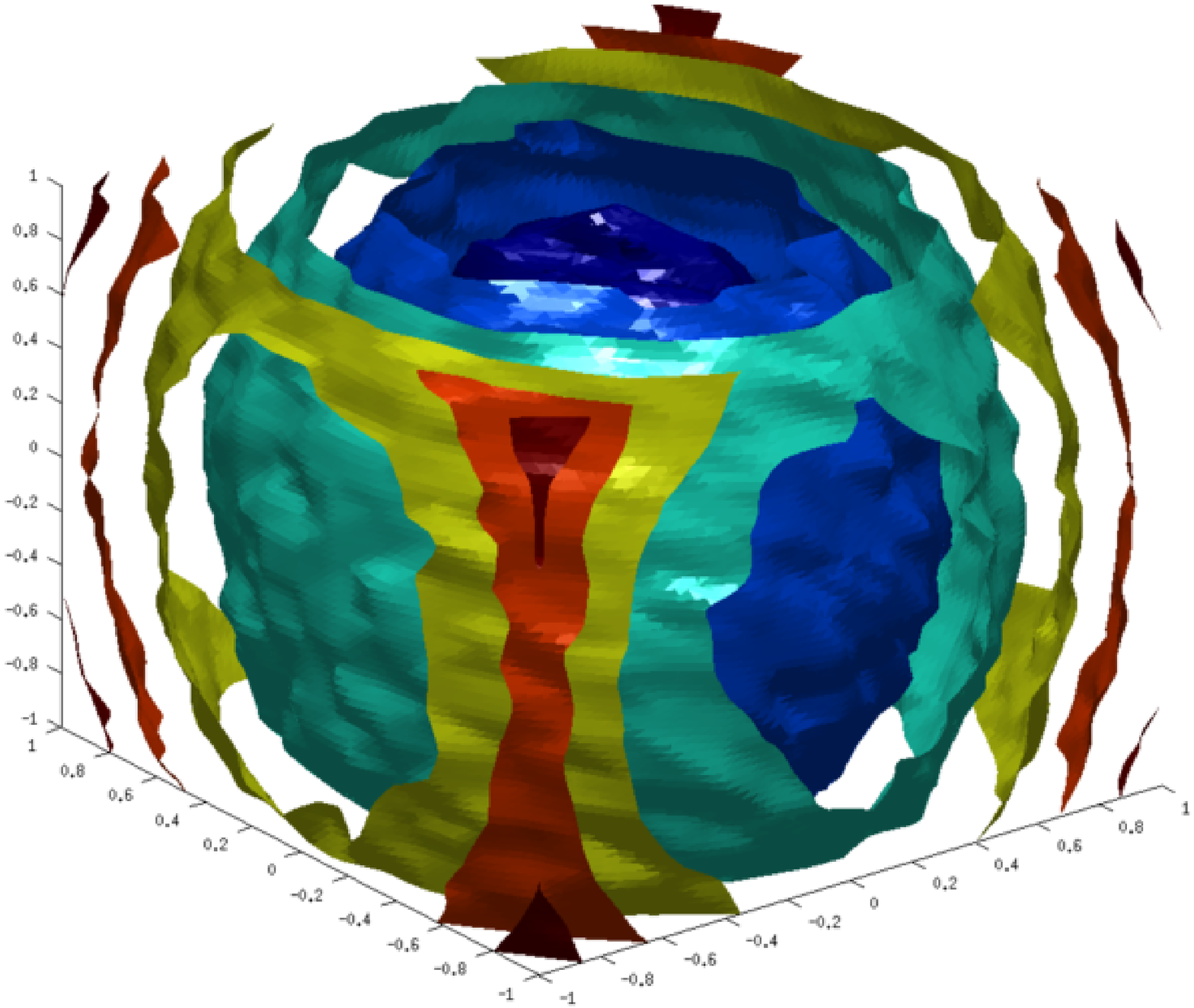}
&\includegraphics[height=\widthh cm]{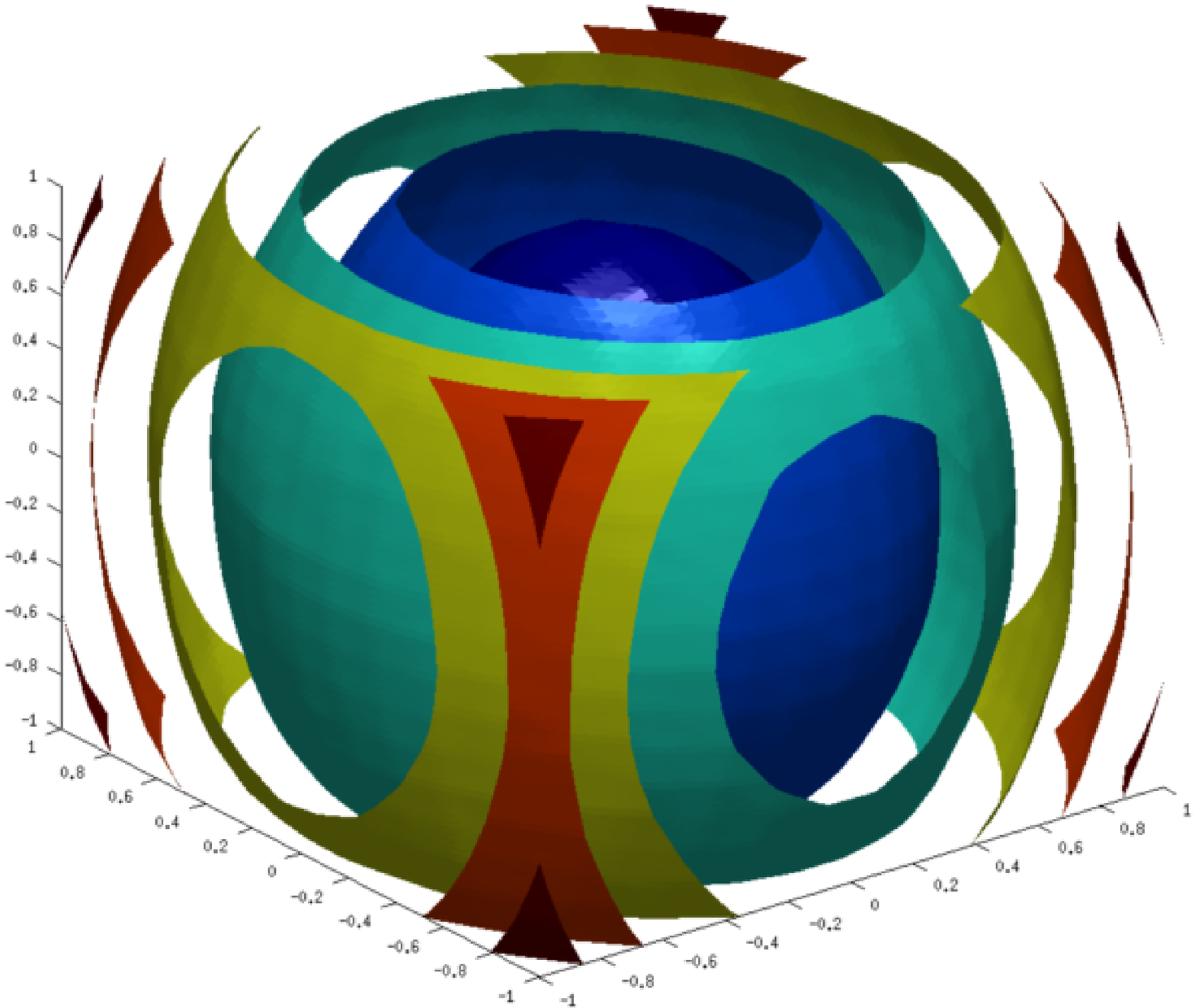}\\
\end{tabular}
\caption{Denoising a convex graph by one dimensional projections.}
\label{fig:conv1}
\end{figure}
 
As described in previous sections, our discretization of the space of
convex functions is not interior. However, thanks to Theorem
\ref{th:dh}, we obtain a converging discretization process that uses
fewer constraints than previously proposed interior approaches.  More
explicitly, we illustrate below our method on the following
three-dimensional denoising setting.  Let $u_0(x,y,z) = \frac{x^2}{3}
+ \frac{y^2}{4} + \frac{z^2}{8}$, $X = [-1,1]^3$ and set $c =
\frac{1}{40}$. We carried our computation on a regular grid made of
$80^3$ points and we look for an approximation in the space of
piecewise-linear functions. The parameter used to discretize the
convexity constraints is set to $\eps=0.02$. Figure~\ref{fig:conv1}
displays the result of the SDMM algorithm after $10^4$ iterations. This computation took less than five minutes on a standard computer.

To illustrate the versatility of the method, we performed the same
denoising experience in the context of support functions, using the
discretization explained in Section~\ref{sec:generalization}. As in
the previous example, we consider a support function perturbed by
additive Gaussian noise $\hh_0(p) = \hh^*(p) + c {\mathcal
  N}(0,1)$. In the numerical application, $\hh^*$ is the support
function of the unit isocaedron and $c=0.05$, as shown on the left of
Figure~\ref{fig:ico}. Our goal is to compute the projection of
$\hh_0$ to the space of support functions:
$$\min_{\hh \in \mathcal \H^s} \nr{\hh - \hh_0}_{L^2( \Sph^{d-1})}^2.$$
In order to relax the constraint $\H^s$, we imposed one dimensional
constraints on a family of $2000$ great circles of $\Sph^{d-1}$
uniformly distributed and a step discretization of every circular arc
equal to $0.02$. We obtained a very satisfactory reconstruction of
$\hh_i$ after $10^4$ iterations of the SDMM algorithm, as displayed on
the right of Figure~\ref{fig:ico}.

\begin{figure}[t]
\centering
\includegraphics[width=6cm]{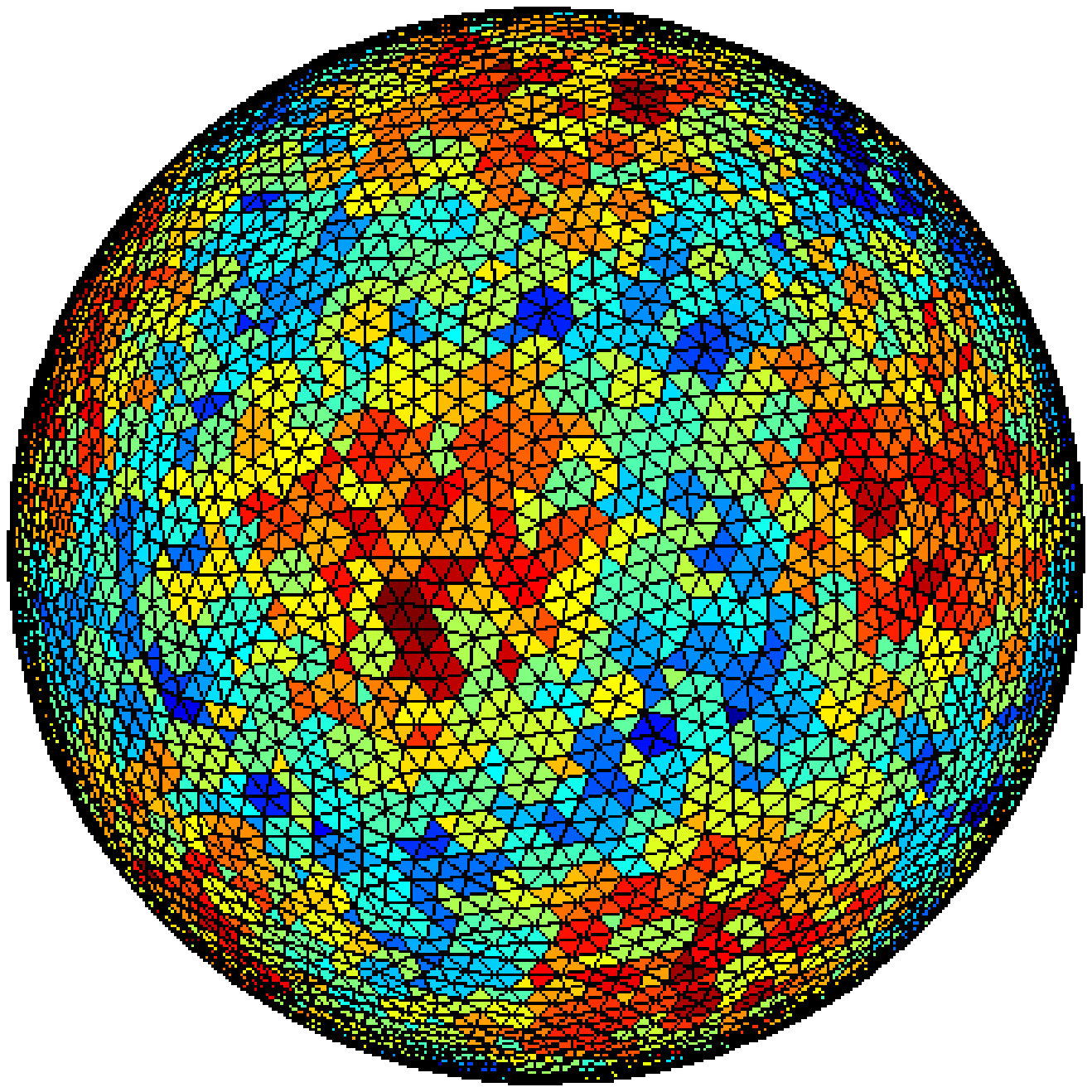}\hfill
\includegraphics[width=6cm]{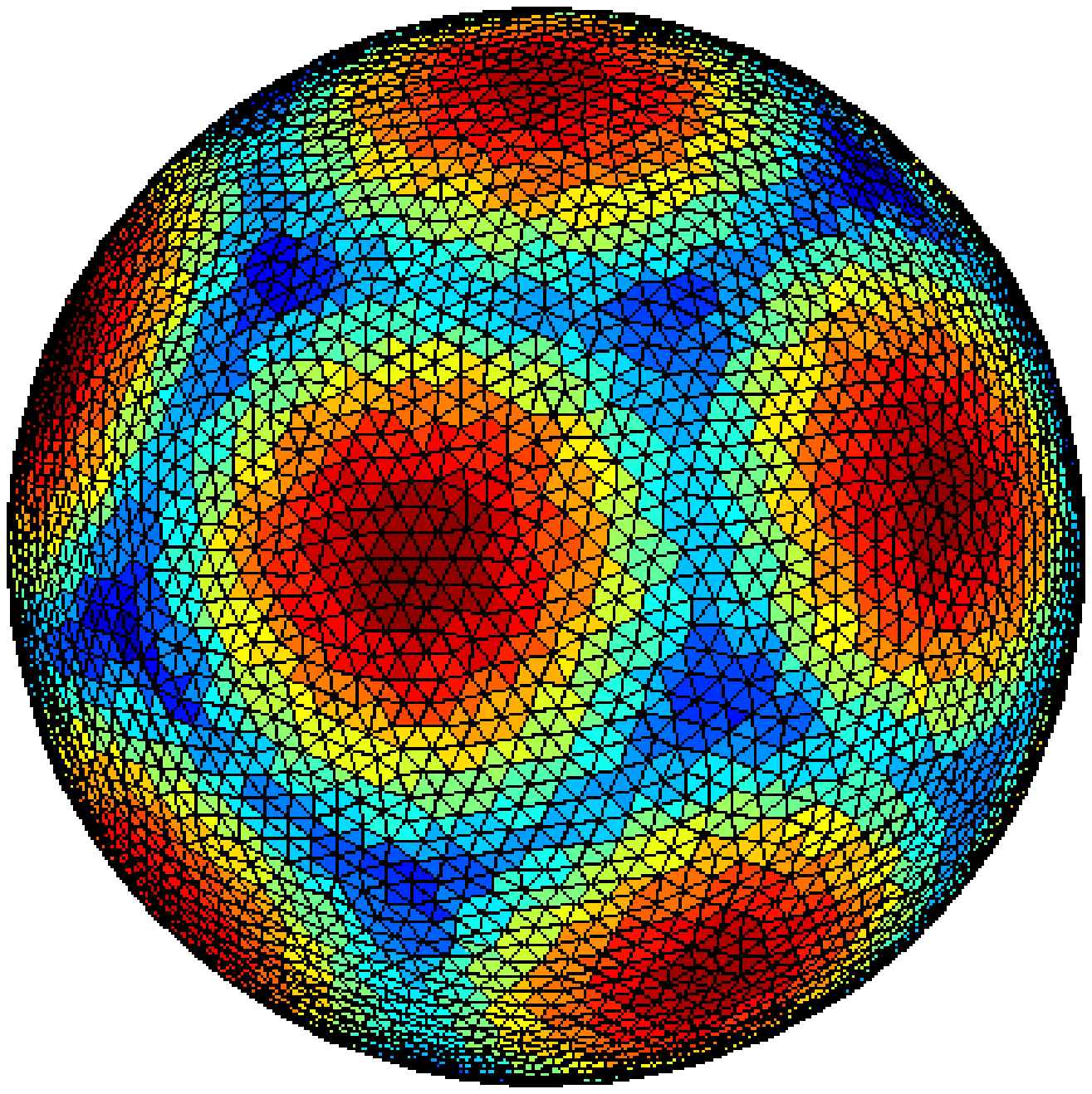}
\caption{Denoising the support function of a convex body. On the left the perturbed support function of the icosaedron. On the right its projection into the set of support functions.}
\label{fig:ico}
\end{figure}

\section{Application II : Principal-agent problem}
\label{pagent_section}
The principal-agent problem formalizes how a monopolist selling
products can determine a set of prices for its products so as to
maximize its revenue, given a distribution of customer -- the
agents. We describe the simplest geometric version of this problem in
the next paragraph. Various instances of this problem are then used as
numerical benchmarks for our relaxed convexity constraints.

\subsection{Geometric principal agent problem}
Let $X$ be a bounded convex domain of $\Rsp^d$, a distribution of
agent $\rho: X \to \Rsp$ and a finite subset $K\subseteq
X$. The \emph{monopolist} or \emph{principal} needs to determine
a \emph{price menu} $\pi$ for pick-up or deliveries, so as to maximize
its revenue. The principal has to take into account the two following
constraints: (i) the agents will try to maximize their utility and
(ii) there is a finite subset $K\subseteq X$ of facilities, that
compete with the principal and force him to set its price $\pi(y)$ to
zero at any $y$ in $K$.  For a given price menu $\pi$, the utility of
a location $y$ for an agent located at a position $x$ in $X$ is
given by $u_\pi(x,y) = - \frac{1}{2} \nr{x-y}^2 - \pi(y)$. The fact
that each agent tries to maximize his utility means that he will
choose a location that balances closeness and price. The maximum
utility for an agent $x$ is given by:
$$u_\pi(x) := \max_{y\in X} u(x,y) = - \frac{1}{2}\nr{x}^2 +
\max_{y \in X} \left[\sca{x}{y} - \frac{1}{2}\nr{y}^2 -
  \pi(y)\right]$$ Let us denote $\bar{u}_\pi(x)$ the convex function
$u_\pi(x) + \frac{1}{2} \nr{x}^2$. This function is differentiable
almost every point $x$ in $X$, and at such a point the gradient
$\nabla \bar{u}_\pi(x)$ agrees with the best location for $x$, i.e. 
$\nabla \bar{u}_\pi(x) = \arg\max_{y} u(x,y)$.
This implies the following equality:
$$\bar{u}_\pi(x) = \sca{x}{\nabla \bar{u}_\pi(x)} - \frac{1}{2}\nr{\nabla
\bar{u}_\pi(x)}^2 - \pi(\nabla \bar{u}_\pi(x))$$

Our final assumption is that the cost of a location for the principal
is constant. Our previous discussion implies that the total revenue of
the principal, given a price menu $\pi$, is computed by the following
formula
\begin{align}
R(\pi) &= \int_{X} \pi(\nabla \bar{u}_\pi(x)) \rho(x) \dd x
\notag \\ &= -\int_{X} \left[ \bar{u}_\pi(x) - \sca{x}{\nabla
    \bar{u}_\pi(x)} + \frac{1}{2}\nr{\nabla \bar{u}_\pi(x)}^2 \right] \rho(x) \dd x
\end{align}
Changing the unknown from $\pi$ to $v := \bar{u}_\pi$, the assumption
that the price vanishes on the set $K$ translates as $u_\pi\geq
\max_{y\in K} - \frac{1}{2} \nr{\cdot-y}^2$ or equivalently
$$v(x) = \bar{u}_\pi(x)\geq \max_{y\in K} \sca{x}{y} -
\frac{1}{2}\nr{y}^2.$$ Thus, we reformulate  the principal's problem
in term of $v$ as the minimization of the following functional:
\begin{equation}
 L(v) := \int_{X} \left[ v(x) + \frac{1}{2}\nr{\nabla v(x) -
    x}^2 \right] \rho(x) \dd x
\label{eq:geo}
\end{equation} where the maximum is taken over the
set of convex functions $v:X\to\Rsp$ that satisfy the lower bound
$v \geq \max_{y\in K} \sca{.}{y} - \frac{1}{2}\nr{y}^2$.

\subsection{Numerical results} We present three numerical
experiments. The first one concerns a linear variant of the
principal-agent problem. The second and third one concern the
geometric principal-agent problem presented above: we maximize the
functional $L$ of Equation~\eqref{eq:geo} over the space of
non-negative convex functions, with $X = \B(0,1)$ and $X=[1,2]^2$
respectively, $\rho$ constant and $K = \{(0,0)\}$.

\begin{table}
\centering
\begin{tabular*}{\textwidth}{c @{\extracolsep{\fill}} ccccc}
\# points & $\eps$   &  $|L - L_{\mathrm{opt}}|$ & $\nr{u - u_{\mathrm{opt}}}_\infty$  & CPU  \\ 
\hline
$30\times 30$ & $0.06$ & $1.3\cdot 10^{-3}$  & $2 \cdot 10^{-3}$  & 11s  \\ 
$60\times 60$ & $0.03$ & $9.8\cdot 10^{-4}$  & $1.6 \cdot 10^{-3}$  & 251s  \\ 
$90\times 90$ & $0.02$ & $9.7\cdot 10^{-4}$  & $1.1 \cdot 10^{-3}$  & 500s  \\ 
\end{tabular*}
\smallskip
  \caption{Convergence of numerical approximations for the geometric principal-agent problem
 (radial case).\label{table-pagentradial}}
\end{table}

\begin{table}
\centering
\begin{tabular*}{\textwidth}{c @{\extracolsep{\fill}} ccccc}
\# points & $\eps$   &  $|M - M_{\mathrm{opt}}|$ & $\nr{u - u_{\mathrm{opt}}}_\infty$  & CPU  \\ 
\hline
$900$ & $0.06$ & $8\cdot 10^{-5}$  & $1.15 \cdot 10^{-2}$  & 11s  \\ 
$3600$ & $0.03$ & $8\cdot 10^{-5}$  & $1.00 \cdot 10^{-2}$  & 28s  \\ 
$8100$ & $0.02$ & $4.3\cdot 10^{-5}$  & $8.46 \cdot 10^{-3}$  & 87s  \\ 
\end{tabular*}
\smallskip
\caption{Convergence of numerical approximations for the linear principal-agent.\label{table-pagentlinear}\label{tab:linear}}
\end{table}

\subsubsection{Linear principal agent.} As a first benchmark, we
consider a variant of the principal-agent problem where the minimized
functional is linear in the utility function $u$ \cite{manelli2007multidimensional}. The goal is
to minimize the following functional
$$ M(u) := \int_{X} (u(x)-\sca{\nabla u}{x})\rho(x)\dd x, $$
where $X=[0,1]^2$ and $\rho = 1$, over the set of convex functions
whose gradient is included in $[0,1]^2$. The solution to this problem
is known explicitely:
$$u_{\mathrm{opt}}(x_1,x_2) = \max\{0,x_1-a,x_2-a,x_1+x_2-b\}$$
where $a = 2/3$ and $b=(4-\sqrt{2})/3$. We solve the linear
principal-agent problem on a regular grid meshing $[0,1]^2$, and
compare it to the exact solution on the grid
points. Table~\ref{tab:linear} displays the numerical results for
various grid sizes and choices of $\eps$.

\subsubsection{Geometric principal agent, radial case} In order to
evaluate the accuracy of our algorithm, we first solve the
(non-linear) geometric principal-agent problem on the unit disk, with
$K = \{(0,0)\}$ and $\rho$ constant. The optimal profile is radial in
this setting, and one can obtain a very accurate description of the
optimal radial component by solving a standard convex quadratic
programming problem. In parallel, we compute an approximation of the
$2$D solution on an unstructured mesh of the disk. On the left of
Figure \ref{fig:pagent}, we show that our solution matches the line of
the one dimensional profile after $10^3$ iterations of the SDMM
algorithm, for $\delta = 0.12$ and $\eps = 1/50$. Table
\ref{table-pagentradial} shows the speed of convergence of our method,
both in term of computation time and accuracy, with $10^3$ iterations.

\subsubsection{Geometric principal agent, Rochet-Choné case} We
recover numerically the so-called bunching phenomena predicted by
Rochet and Choné \cite{rochet1998ironing} when $X=[1,2]^2$, $\rho$ is
constant and $K=\{(0,0)\}$, thus confirming numerical results from
\cite{ekeland2010algorithm,mirebeau2013,aguilera2008approximating}.
On the right of Figure \ref{fig:pagent}, we show the numerical
solution defined on a regular mesh of the square of size $60 \times
60$, with $\eps = 0.02$. In this computation, the interpolation
operator is constructed using P$3$ finite elements, so as to
illustrate the flexibility of our method.

\begin{figure}
\centering
\includegraphics[width=6.3cm]{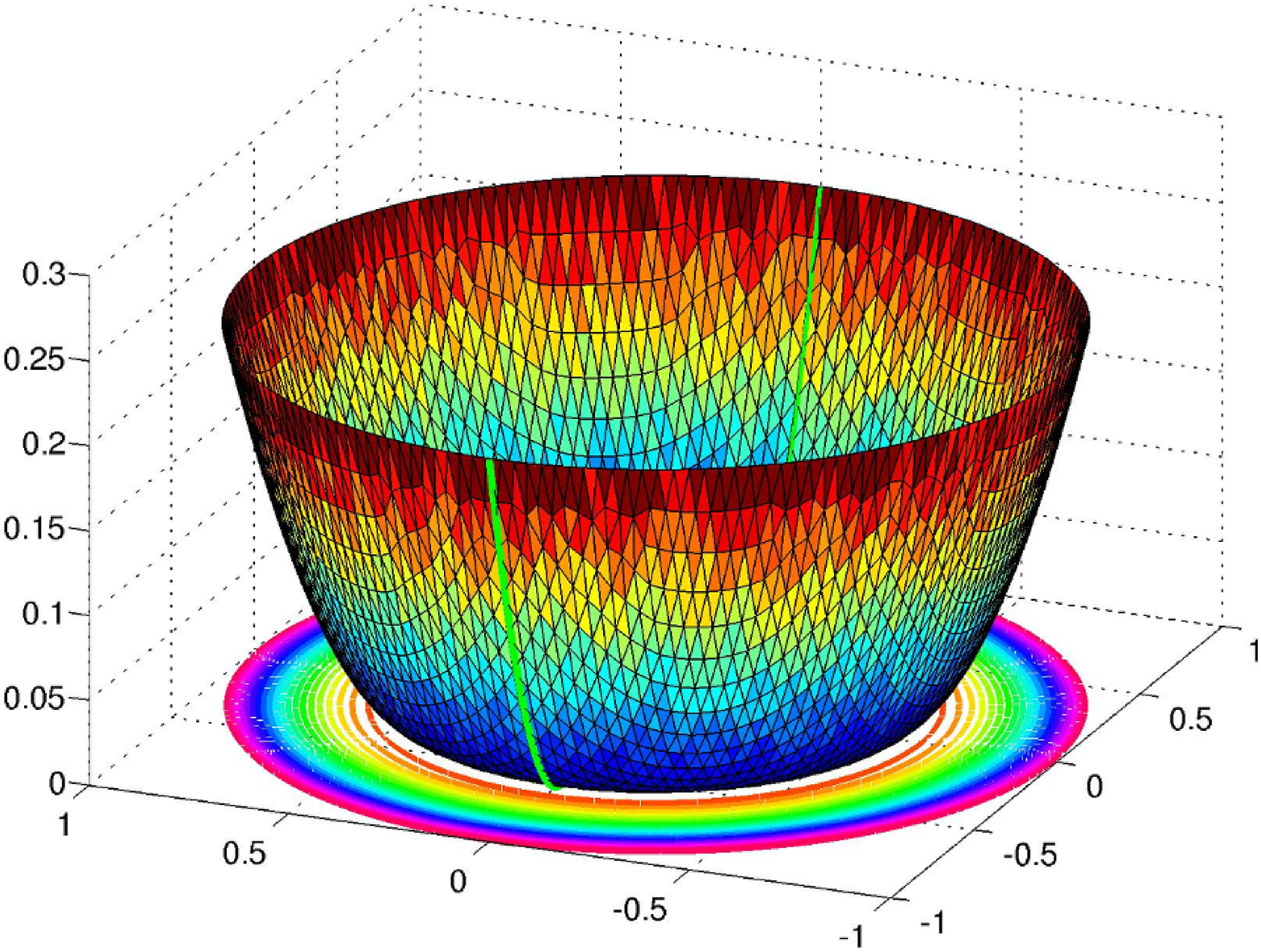}\hfill
 \includegraphics[width=6.3cm]{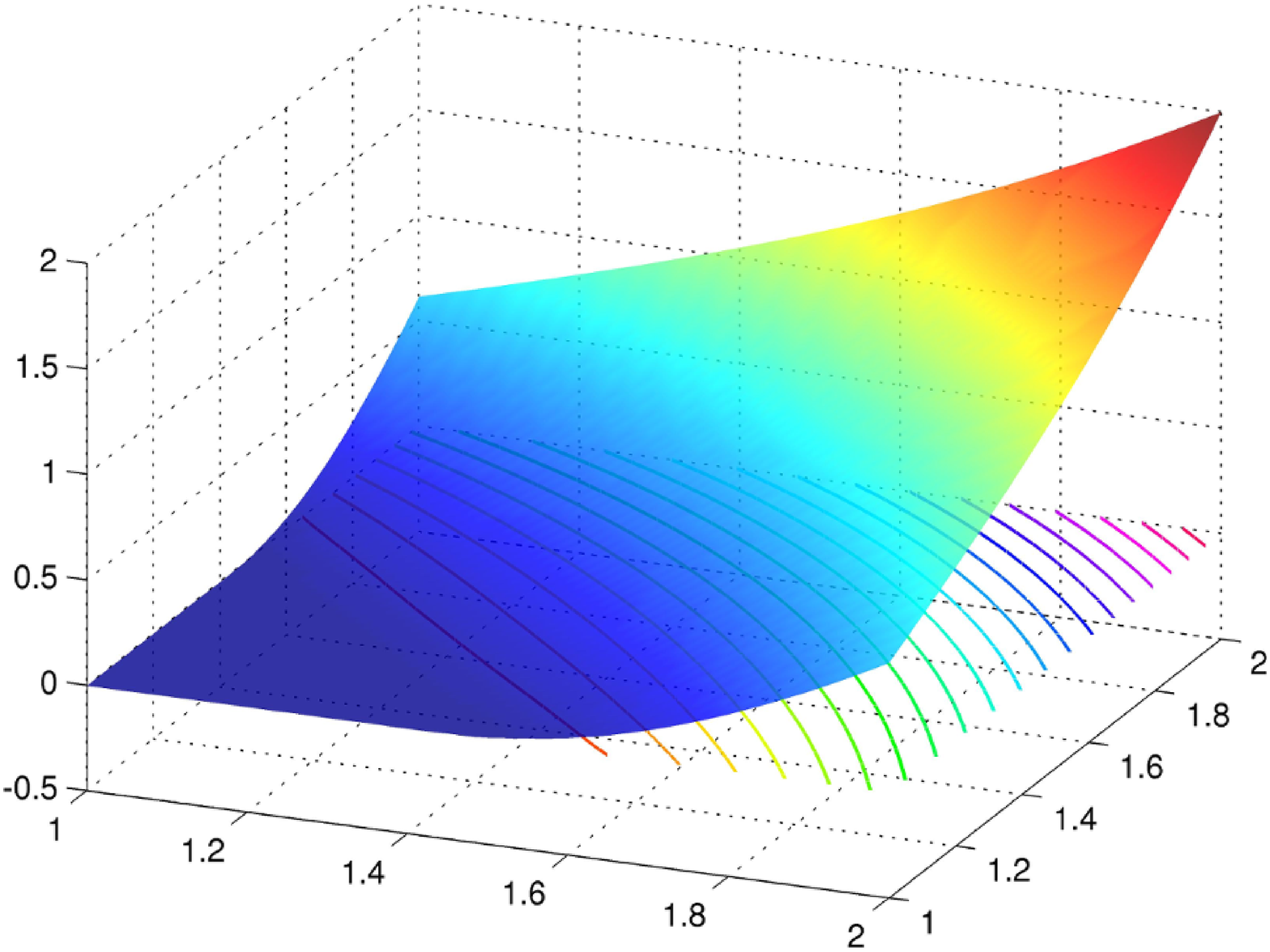} 
 \caption{\emph{(Left)} Numerical approximation to the principal-agent
   problem on $X=\B(0,1)$, with $\delta = 1/60$ and $\eps = 1/50$. The
   profile of the $1$D solution is reported as a bold line on the
   graph. \emph{(Right)} Numerical approximation to the
   principal-agent problem with $X=[1,2]^2$ by P$3$ finite elements.}
\label{fig:pagent}
\end{figure}

\section{Application III : Closest convex set with constant width}
\label{sec:constant}

A convex compact set $K$ of $\Rsp^d$ has constant width $\alpha>0$ if
all its projection on every straight line are segments of length
$\alpha$. This property is equivalent to the following constraints on
the support function of $K$ :
\begin{equation}
  \label{widthconstr}
\forall \nu \in  \Sph^{d-1},\, \hh_K(\nu)+\hh_K(-\nu)=\alpha.
\end{equation}
Surprisingly, balls are not the only bodies having this property. In
dimension two for instance, Reuleaux's triangles, which are obtained
by intersecting three disks of radius $\alpha$ centered at the
vertices of an equilateral triangle have constant width
$\alpha$. Moreover, Reuleaux's triangles have been proved by Lebesgue
and Blaschke to minimize the area among two-dimensional bodies with prescribed constant width.

In dimension three, this problem is more difficult. Indeed the mere
existence of non trivial three-dimensional bodies of constant width is
not so easy to establish. In particular, no finite intersection of
balls has constant width, except balls themselves
\cite{oudet2007bodies}. As a consequence and in contrast to the two
dimensional case, the intersection of four balls centered at the
vertices of a regular simplex is not of constant width. In 1912,
E. Meissner described in \cite{meissner3} a process to turn this
spherical body into an asymmetric bodies with constant width, by
smoothing three of its circular edges.  This famous body is called
``Meissner tetrahedron'' in the literature
\cite{kawohl2011meissner}. It is suspected to minimize the volume
among three dimensional bodies with the same constant width. Let us
point out that Meissner construction is not canonical in the sense
that it requires the choice the set of three edges that have to be
smoothed. As a consequence, there actually exists two kinds of
``Meissner tetrahedron'' having the same measure.

In these two constructions, the regular simplex seems to play a
crucial role in the optimality (see also \cite{jin2012asymmetry} for a
more rigorous justification of this intuition). It is therefore
natural to search for the body with constant width that is the closest
to a regular simplex. In an Hilbert space, the projection on a convex
set is uniquely defined. Thus, the Meissner tetrahedra cannot be
obtained as projections of a regular simplex to the convex set
$\H^s\cap \mathcal{W}$ with respect to the $\LL^2$ norm between
support functions. Such an obstruction does not hold for the $\LL^1$
and $\LL^\infty$ norm, which are not strictly convex.  We illustrate
below that our relaxed approach can be used to numerically investigate
these questions. The optimization problem that we have to approximate
is 
$$\min_{\hh \in \H^s \cap \mathcal{W}} \nr{\hh_0 -
  \hh}_{\LL^p(\Sph^{2})},~1\leq p \leq \infty$$ where $\mathcal{W}$ is
the set of function of $\Sph^{2}$ which satisfy the width
constraints~\eqref{widthconstr}.

As explained is Section~\ref{sec:generalization}, we relax the
constraint of being a support function, by imposing convexity-like
conditions on a finite family of great circles of the sphere. In the
experiments presented below the number of vertices in our mesh of
$\Sph^{2}$ is $5000$. We choose a family of $2000$ great circles of
$\Sph^{2}$ uniformly distributed (with respect to their normal
direction) and a step discretization of every circular arc equal to
$0.02$. Finally, the constant width constraint $\mathcal{W}$ is
approximated by imposing that antipodal values of the mesh must
satisfy a set of linear equality constraints, which can be easily
implemented in the proximal framework depicted in Section~\ref{prox}.
Note that in this first experience, the value of the width constraint
is not imposed.

\begin{figure}[t]
\centering
\label{fig:cw}
\begin{tabular}{r r r}
\includegraphics[height=\widthhs cm]{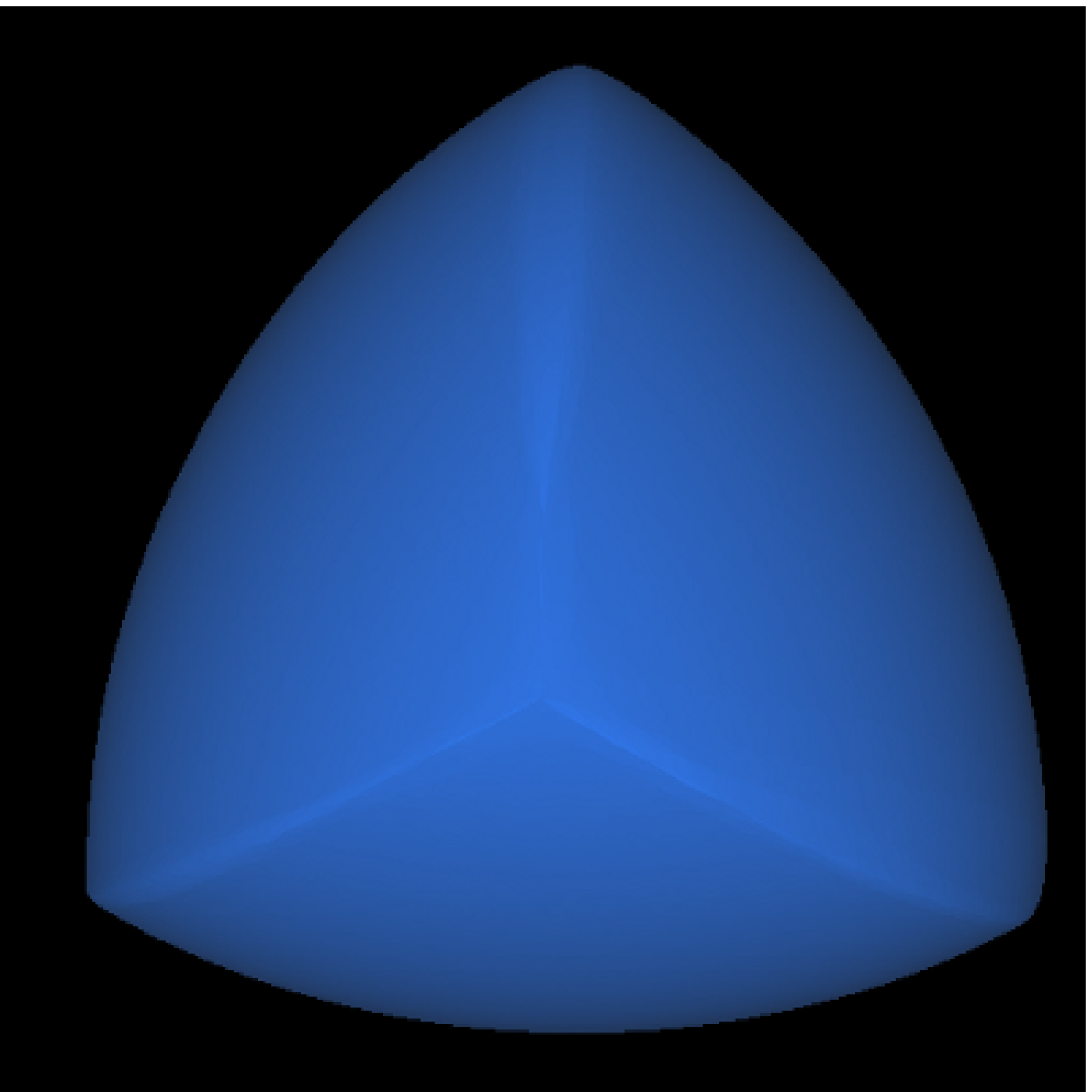}
&\includegraphics[height=\widthhs cm]{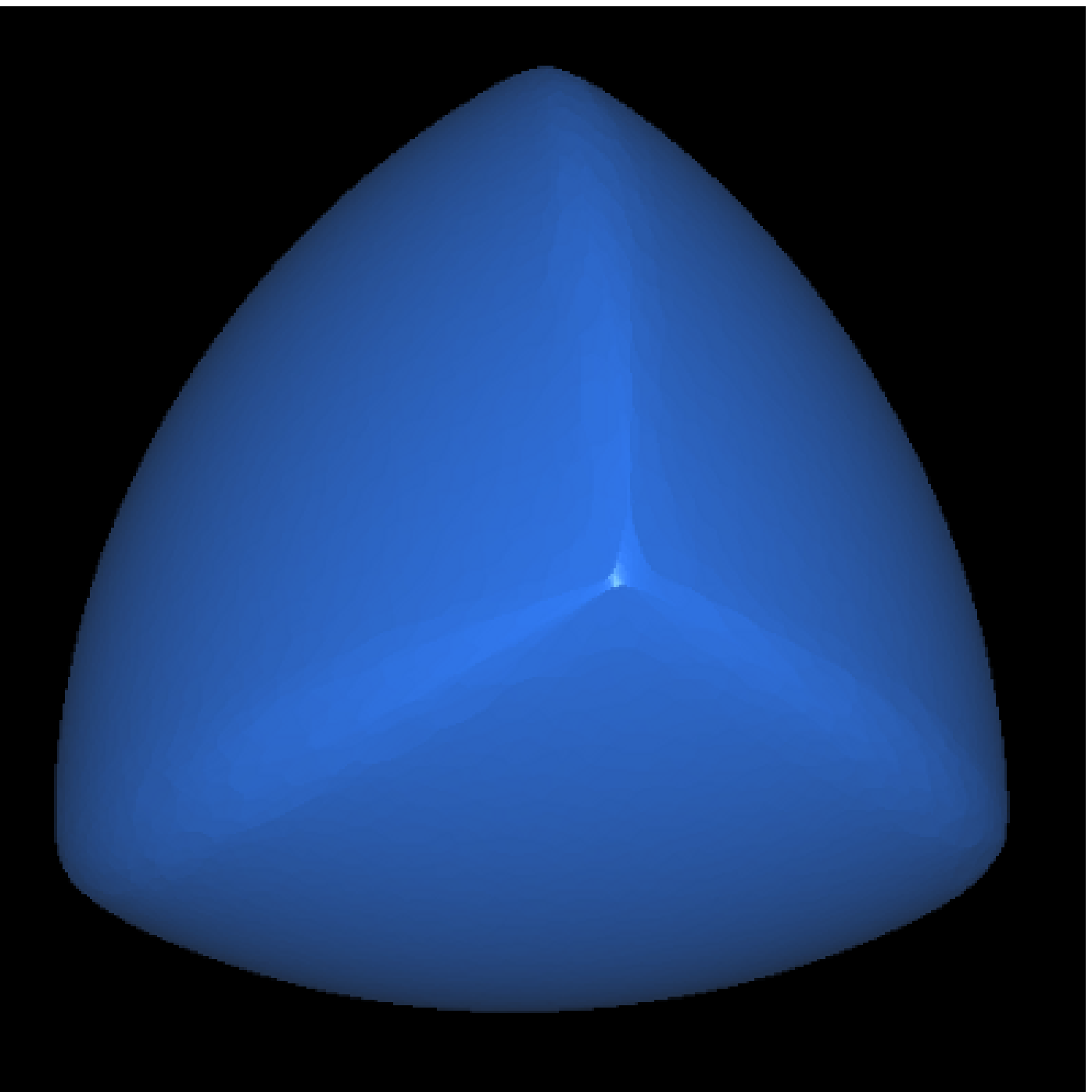}
&\includegraphics[height=\widthhs cm]{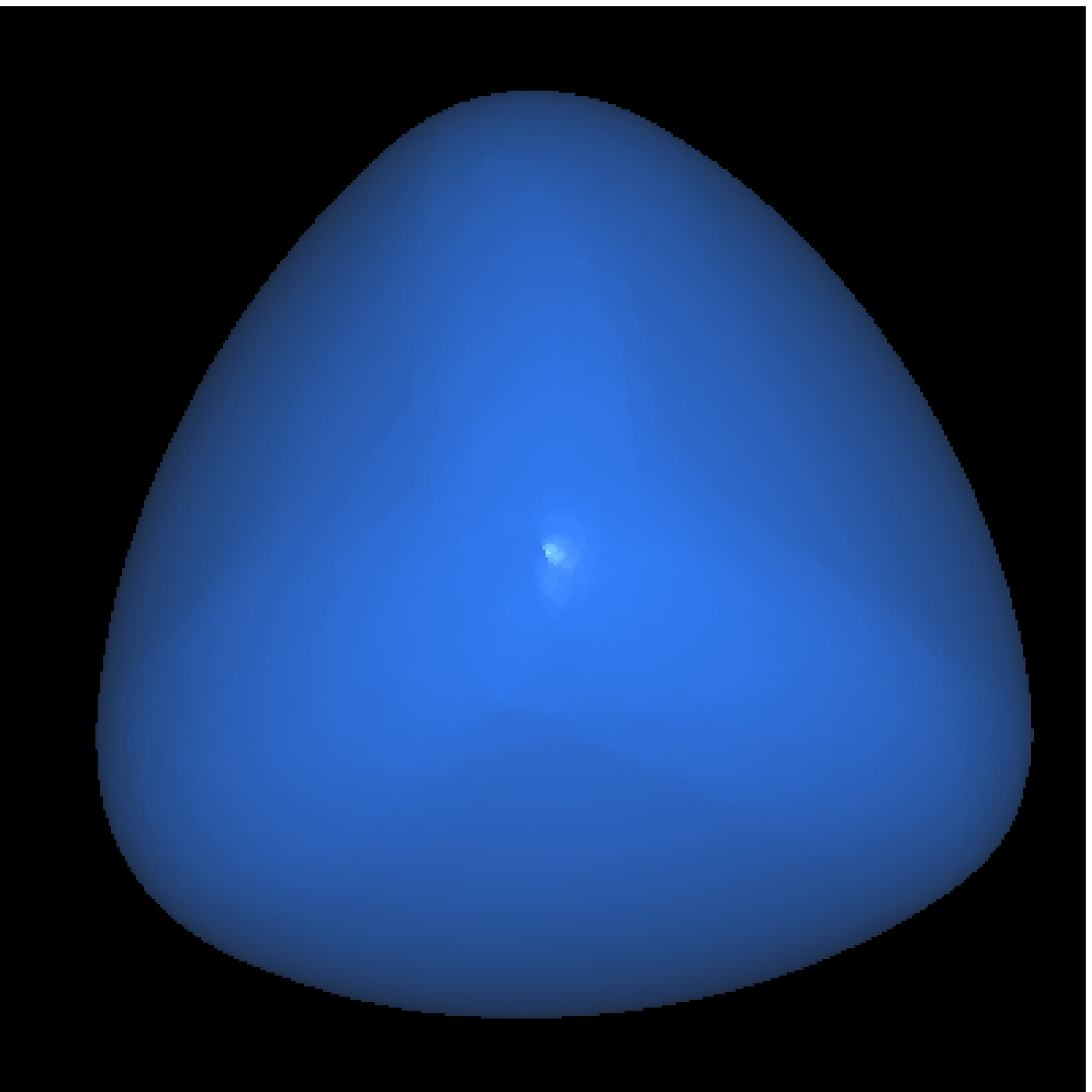}\\
\end{tabular}
\caption{Reconstruction of the convex bodies associated to the $\LL^1$, $\LL^2$ and $\LL^\infty$ projection of $\hh_S$ without prescribing the width value.}
\end{figure}

\begin{table}
\label{table-result_projs}

\begin{tabular}{c|c|c|c|c}
                                             & Surface    & Volume        & Width    & Relative width error   \\ 
$\LL^1$ projection of $\hh_S$      & 2.6616          & 0.36432       & 0.951    &  < 0.001                \\ 
$\LL^2$ projection of $\hh_S$      & 2.5191          & 0.34312       & 0.920    &  < 0.003                \\ 
$\LL^\infty$ projection of $\hh_S$  & 2.1351          & 0.28081       & 0.835    &  < 0.001   
\end{tabular}
\smallskip

\caption{Numerical results for the projections of $\hh_S$}
\vspace{-.5cm}
\end{table}

We present in Table~\ref{table-result_projs} and Figure~\ref{fig:cw},
our numerical description of the projections of the support function
of a regular simplex in the set of support function of constant width
bodies for the $\LL^1$, $\LL^2$ and $\LL^\infty$ norms. One can
observe that the resulting support functions describe a body with
constant width within an error of magnitude $0.1 \%$. In other words
the gap between the minimal width and the diameter is relatively less
than $0.001$. In the $L^1$ case we obtain a convex body whose surface
area and volume are close to those of a Meissner body of same width,
within a relative error of less than $0.01$.  We also performed the
same experiment starting from the support functions of others platonic
solids. For any of these other solids, and when the value of the width
is not imposed, the closest body with constant seems to always be a
ball.

\bibliographystyle{amsplain}
\bibliography{projconvex}

\end{document}